\documentclass[reqno]{amsart}
\usepackage{amssymb}
\usepackage{upgreek}
\usepackage{xcolor}
\newtheorem{theorem}{Theorem}[section]
\newtheorem{proposition}[theorem]{Proposition}
\newtheorem{lemma}[theorem]{Lemma}

\newtheorem{corollary}[theorem]{Corollary}
\theoremstyle{definition}

\newtheorem{remark}[theorem]{Remark}

\numberwithin{equation}{section}
\allowdisplaybreaks
\begin{document}
\title[]
{Small solutions to inhomogeneous and homogeneous quadratic congruences modulo prime powers}

\author{Stephan Baier}
\address{Ramakrishna Mission Vivekananda Educational and Research Institute, Belur Math, Howrah, West Bengal-711202, INDIA}
\curraddr{}
\email{stephanbaier2017@gmail.com}

\author{Arkaprava Bhandari}
\address{Ramakrishna Mission Vivekananda Educational and Research Institute, Belur Math, Howrah, West Bengal-711202, INDIA}
\curraddr{}
\email{arkapravabhandari@gmail.com}

\author{Anup Haldar}
\address{Department of Mathematics, A. P. C. Roy Government College, Himachal Bihar, Matigara, Siliguri-734010, West Bengal}
\curraddr{}
\email{anuphaldar1996@gmail.com}

\subjclass[2020]{11D09,11D45,11D79,11E12,11L05,11L40}

\keywords{quadratic congruences, small solutions, Gauss sums, Kloosterman sums, Sali\'e sums, distribution of modular square roots, representation by quadratic forms}

\maketitle
\begin{abstract}
We prove asymptotic formulae for small weighted solutions of quadratic congruences of the form $\lambda_1x_1^2+\cdots +\lambda_nx_n^2\equiv \lambda_{n+1}\bmod{p^m}$, where $p$ is a fixed odd prime, $\lambda_1,...,\lambda_{n+1}$ are integer coefficients such that $(\lambda_1\cdots \lambda _{n},p)=1$ and $m\rightarrow \infty$. If $n\ge 6$, $p\ge 5$ and the coefficients are fixed and satisfy $\lambda_1,...,\lambda_n>0$ and $(\lambda_{n+1},p)=1$ (inhomogeneous case), we obtain an asymptotic formula which is valid for integral solutions $(x_1,...,x_n)$ in cubes of side length at least $p^{(1/2+\varepsilon)m}$, centered at the origin.  If $n\ge 4$ and $\lambda_{n+1}=0$ (homogeneous case), we prove a result of the same strength for coefficients $\lambda_i$ which are allowed to vary with $m$. These results extend previous results of the first- and the third-named authors and N. Bag. 
\end{abstract} 

\tableofcontents

\section{Introduction and statement of results}
The existence of small solutions of congruences of the form
\begin{equation} \label{gene}
F(x_1,...,x_n)\equiv \lambda \bmod{q},
\end{equation}
where $F$ is a form in $n$ variables, has received a lot of attention. Of particular interest is the case of diagonal forms 
$$
F(x_1,...,x_n)=\lambda_1x_1^k+\cdots+ \lambda_nx_n^k.
$$
A very general result in this direction was obtained by Cochrane, Ostergaard and Spencer \cite{CoOsSp} who established the existence of a solution to \eqref{gene} for this case in any cube of sidelength $q^{1/k+\varepsilon}$ if $q$ is a prime, $k\ge 2$, 
$0<\varepsilon< 1/(k(k-1))$ and $n>(k-1)/\varepsilon$. Their paper \cite{CoOsSp} also provides a history of results on this problem. The homogeneous case when $\lambda\equiv 0\bmod{q}$ is generally easier to handle and has therefore received more attention than the inhomogeneous case. Particularly strong results have been established for homogeneous quadratic congruences ($\lambda=0$ and $k=2$) in at least three variables ($n\ge 3$) by Schinzel, Schlickewei and Schmidt \cite{SchSchSch}, Heath-Brown \cite{Hea1,Hea2,Hea3} and Hakimi \cite{Hak}. In the present article, we are interested in asymptotic formulae for the number of small solutions of quadratic congruences (rather than just existence results), with a special emphasis on the more difficult inhomogeneous case. We will confine ourselves to prime power moduli $q=p^m$. This allows us to evaluate Kloosterman and Sali\'e sums coming up in our method explicitly and use Hensel-type arguments. Below we review recent research on asymptotic formulae for the number of small solutions of {\it homogeneous} quadratic congruences to prime power moduli.   
 
In \cite{BaHa}, the first- and third-named authors considered small solutions of homogeneous ternary diagonal quadratic congruences to prime power moduli. They proved the following asymptotic result for forms with fixed coefficients. 

\begin{theorem}[Theorem 1 in \cite{BaHa}] \label{mainresult1BaHa}
Fix $\varepsilon>0$, a prime $p\ge 3$ and integers $\lambda_1,\lambda_2,\lambda_3$ such that $(\lambda_1\lambda_2\lambda_3,p)=1$.  Set 
$$
C_p(\lambda_1,\lambda_2,\lambda_3):=\frac{(p-s_p(\lambda_1,\lambda_2,\lambda_3))(p-1)}{p^2},
$$
where
\begin{equation} \label{sdef}
s_p(\lambda_1,\lambda_2,\lambda_3):=2+\left(\frac{-\lambda_1\lambda_2}{p}\right)+
\left(\frac{-\lambda_1\lambda_3}{p}\right)+\left(\frac{-\lambda_2\lambda_3}{p}\right).
\end{equation}
Let $\Phi:\mathbb{R}\rightarrow \mathbb{R}_{\ge 0}$ be a Schwartz class function.
Then as $m\rightarrow \infty$, we have the asymptotic formula
\begin{equation} \label{main}
\sum\limits_{\substack{(x_1,x_2,x_3)\in \mathbb{Z}^3\\ (x_1x_2x_3,p)=1\\ \lambda_1x_1^2+\lambda_2x_2^2+\lambda_3x_3^2 \equiv 0 \bmod{q}}} \Phi\left(\frac{x_1}{N}\right)
\Phi\left(\frac{x_2}{N}\right)\Phi\left(\frac{x_3}{N}\right)\sim
\hat{\Phi}(0)^3\cdot C_p(\lambda_1,\lambda_2,\lambda_3)\cdot \frac{N^3}{q}
\end{equation}
with $q=p^m$, 
provided that $N\ge q^{1/2+\varepsilon}$ and $p>s_p(\lambda_1,\lambda_2,\lambda_3)$.
\end{theorem}

In \cite{Hal}, the third-named author extended the above result to non-diagonal homogeneous ternary quadratic forms. In \cite{BaBaHa}, this result was further extended to homogeneous quadratic forms in at least 3 variables by N. Bag and the first- and third-named authors, with a slightly different condition in place of $(x_1\cdots x_n,p)=1$. They proved the following.

\begin{theorem}[Theorem 2 in \cite{BaBaHa} for $(x_{0,1},...,x_{0,n})=(0,...,0)$)]\label{MT2}
Fix $\varepsilon>0$, $n\in\mathbb{N}$ with $n\ge 3$ and a prime $p\ge 3$. 
Let $(a_{i,j})_{1\le i,j\le n}$ be a fixed symmetric integral matrix such that the quadratic form  
$$
Q(x_1,...,x_n):=\sum\limits_{i,j=1}^n a_{i,j}x_ix_j
$$
is non-singular modulo $p$. Set $A_p(Q):=\sharp\mathcal{A}(Q)/p^{n-1}$, where
\begin{align*}
\mathcal{A}(Q)&:=\left\{(x_1,...,x_{n})\in\mathbb{Z}^n:~0\leq x_1,...,x_{n}\leq p-1,\ (x_1,...,x_n)\not=(0,...,0), \ Q(x_1,...,x_n)\equiv 0 \bmod p\right\}.
\end{align*}
Let $\Phi:\mathbb{R}\rightarrow \mathbb{R}_{\ge 0}$ be a Schwartz class function.
Then as $m\rightarrow \infty$, we have the asymptotic formula
\begin{equation*}
\sum\limits_{\substack{(x_1,...,x_{n})\in \mathbb{Z}^{n}\\ (x_1,...,x_n)\not\equiv (0,...,0)\bmod p\\ Q(x_1,...,x_n)\equiv 0 \bmod{q}\\
}}
\prod_{i=1}^{n}\Phi\left(\frac{x_i}{N}\right)\sim A_p(Q)\cdot
\hat{\Phi}(0)^{n}\cdot \frac{N^{n}}{q}
\end{equation*}
with $q=p^m$, 
provided that $N\ge  q^{1/2+\epsilon}$.
\end{theorem}

We note that for fixed coefficients $a_{i,j}$, the asymptotic behaviour of the number of solutions changes around the point $N=q^{1/2}$: If $N\le q^{1/2-\varepsilon}$ with $q=p^m$ large enough, then the above congruence turns into the equation 
$$
Q(x_1,x_2,...,x_{n})= 0
$$ 
since the squares of the variables are much smaller than the modulus $q$. In this case, the number of solutions in question is zero if $Q$ is definite and satisfies an asymptotic of the form $\sim C_p N^{n-2}\log N$ if $Q$ is indefinite, where $C_p$ is a suitable constant depending on $p$.

For the situation when the coefficients are allowed to vary with the modulus $q$, the first- and third-named authors proved the following for homogeneous ternary diagonal quadratic forms in \cite{BaHa}.

\begin{theorem}[Theorem 2 in \cite{BaHa}] \label{mainresult2BaHa}
Let the conditions in Theorem \ref{mainresult1BaHa} be kept except that $\lambda_1,\lambda_2,\lambda_3$ are no longer fixed but allowed to vary with $m$. Also suppose that $p>s_p(\lambda_1,\lambda_2,\lambda_3)$. Then the asymptotic formula \eqref{main} holds if $N\ge q^{11/18+\varepsilon}$.
\end{theorem}

This should be compared to Heath-Brown's result in \cite[Theorem 1]{Hea3} for composite moduli $q$ and general quadratic forms in three variables, which establishes the existence of a non-trivial solution of height $\max\{|x_1|,|x_2|,|x_3|\}\ll q^{5/8+\varepsilon}$. 

The goals of the present article are two-fold. Firstly, we aim to prove a variant of Theorem \ref{MT2} for {\it inhomogeneous} quadratic forms. This constitutes the larger part of this article. Secondly, we aim to show that Theorem \ref{MT2} remains valid for {\it homogeneous} diagonal forms in more than 3 variables if the coefficients are allowed to vary with the modulus. That is, if $n>3$, the said result holds for {\it arbitrary} coefficients whenever $N\ge q^{1/2+\varepsilon}$. In contrast, for {\it ternary} forms, it appears difficult to improve the exponent $11/18$ in Theorem \ref{mainresult2BaHa} to $1/2$.     

Our first main result in this article is the following. 

\begin{theorem}\label{thm1}
 Fix $\epsilon>0$, $n\in\mathbb{N}$ with $n\ge 6$ and a prime $p\ge 5$. Let  
\begin{equation}
    Q(x_1,\ldots,x_n):=\lambda_1x_1^2+\lambda_2x_2^2+\cdots+\lambda_nx_n^2-\lambda_{n+1}
\end{equation}
be an inhomogeneous quadratic form with fixed integer coefficients $\lambda_1,...,\lambda_{n+1}$ such that $(\lambda_1\cdots \lambda_{n+1},p)=1$ and the homogeneous quadratic form $\lambda_1x_1^2+\cdots +\lambda_nx_n^2$ is positive definite, i.e. $\lambda_i>0$ for $i=1,...,n$.
Set $B_p(Q):=\sharp\mathcal{B}(Q)/p^{n-1}$, where 
$$
\mathcal{B}(Q):=\left\{(x_1,...,x_n)\in\mathbb{Z}^n : 1\leq x_1,...,x_n\leq p-1,\ Q(x_1,...,x_n)\equiv0\bmod p  \right\}.
$$
Let $\Phi:\mathbb{R}\to\mathbb{R}_{\geq0} $ be a Schwartz class function whose Fourier transform is a bump function. 
Then as $m\to\infty$, we have the asymptotic formula 
    \begin{equation*}
        \sum_{\substack{(x_1,...,x_n)\in\mathbb{Z}^n\\(x_1x_2\cdots x_n,p)=1\\Q(x_1,...,x_n)\equiv 0\bmod p^m}} \prod_{i=1}^n\Phi\left(\frac{x_i}{N}\right)\sim B_p(Q)\cdot \hat{\Phi}(0)^n\cdot\frac{N^n}{q}
    \end{equation*}
with $q=p^m$, provided that $N\geq q^{1/2+\varepsilon}$.  In particular, if $m$ is large enough, then any quadratic congruence $Q(x_1,...,x_n)\equiv 0\bmod{q}$ with $Q$ satisfying the above conditions has a solution $(x_1,...,x_n)$ with $\max\{|x_1|,...,|x_n|\}\le q^{1/2+\varepsilon}$.  
\end{theorem}

We point out that if $k=2$ and $q$ is a prime, then the result of Cochrane, Ostergaard and Spencer mentioned at the beginning of this paper requires the number of variables $n$ to be larger than $1/\varepsilon$ to establish the existence of a non-trivial solution of height $\max\{|x_1|,...,|x_n|\}\le q^{1/2+\varepsilon}$ to the congruence in question. 

Our second main result is the following. 

\begin{theorem} \label{thm1'}
Fix $\varepsilon>0$, $n\in\mathbb{N}$ with $n\ge 4$ and a prime $p\ge 3$. For a diagonal quadratic form
\begin{equation} \label{quadform}
Q(x_1,...,x_n)=\lambda_1x_1^2+\cdots+\lambda_nx_n^2 
\end{equation}
with integer coefficients $\lambda_1,...,\lambda_n$, set $A_p(Q):=\sharp\mathcal{A}(Q)/p^{n-1}$, where
\begin{align*}
\mathcal{A}(Q)&:=\left\{(x_1,...,x_{n})\in\mathbb{Z}^n:~0\leq x_1,...,x_{n}\leq p-1,\ (x_1,...,x_n)\not=(0,...,0), \ Q(x_1,...,x_n)\equiv 0 \bmod p\right\}.
\end{align*}
Let $\Phi:\mathbb{R}\rightarrow \mathbb{R}_{\ge 0}$ be a Schwartz class function.
Then as $m\rightarrow \infty$, for any quadratic form $Q(x_1,...,x_n)$ as given in \eqref{quadform} that is non-singular modulo $p$, i.e. satisfies $(\lambda_1\cdots \lambda_n,p)=1$, we have the asymptotic formula
\begin{equation*}
\sum\limits_{\substack{(x_1,...,x_{n})\in \mathbb{Z}^{n}\\ (x_1,...,x_n)\not\equiv (0,...,0)\bmod p\\ Q(x_1,...,x_n)\equiv 0 \bmod{q}\\
}}
\prod_{i=1}^{n}\Phi\left(\frac{x_i}{N}\right)\sim A_p(Q)\cdot
\hat{\Phi}(0)^{n}\cdot \frac{N^{n}}{q}
\end{equation*}
with $q=p^m$, provided that $N\ge  q^{1/2+\varepsilon}$. In particular, if $m$ is large enough, then any quadratic congruence $Q(x_1,...,x_n)\equiv 0\bmod{q}$ with $Q$ satisfying the above conditions has a solution $(x_1,...,x_n)$ with $\max\{|x_1|,...,|x_n|\}\le q^{1/2+\varepsilon}$.  
\end{theorem}

\begin{remark} (i) In Theorem \ref{thm1}, we have assumed the quadratic form $\lambda_1x_1^2+\cdots+\lambda_nx_n^2$ to be positive definite. It should be possible to drop this condition. However, this would depend on an asymptotic formula for the weighted number of representations by {\it indefinite} quadratic forms which is currently not at our disposal. \medskip\\
(ii) Moreover, in Theorem \ref{thm1}, we have assumed the Fourier transform of our Schwartz class weight function $\Phi$ to be a bump function, which is a commonly used assumption in analytic number theory. We note that we also assumed $\Phi$ to be real-valued and non-negative on $\mathbb{R}$. It is easy to construct a function $\Phi$ with these properties: Take any bump function $f:\mathbb{R}\rightarrow \mathbb{R}_{\ge 0}$ and form the convolution $g=f\ast f^{\ast}$, where $f^{\ast}(x):=f(-x)$. This function $g$ is a bump function as well and its Fourier transform equals $|\hat{f}|^2$. The function $\Phi=|\hat{f}|^2$ is in the Schwartz class and non-negative, and its Fourier transform satisfies $\hat\Phi(x)=g(-x)=g(x)$ and hence is a bump function.  To remove the condition that $\hat{\Phi}$ is a bump function from Theorem \ref{thm1}, we would need to be able to establish a more general version of  \cite[Theorem 1.1]{Jon} (resp., Proposition \ref{modHeath} below) for Schwartz class instead of bump functions. This would require to prove an analogue of \cite[Theorem 5.13]{Jon} for Schwartz class weight functions, which is a non-trivial task. We leave this to future research. 
\medskip\\ (iii) The set $\mathcal{A}(Q)$ in Theorem \ref{thm1'} above is non-empty as a consequence of \cite[parts (iii) and (iv) of Theorem 6 on page 37]{Serre}. We note that under the conditions of our Theorem \ref{thm1'}, $\varepsilon=1=(-1,-1)$, in Serre's notation. \medskip\\
(iv) The set $\mathcal{B}(Q)$ in Theorem \ref{thm1} turns out to be non-empty as well, but we will establish this by a different argument below. We note that in Theorem \ref{thm1}, the condition $(x_1,...,x_n)\not\equiv (0,....0)\bmod{p}$ in Theorem \ref{thm1'} is replaced by the stronger condition $(x_1\cdots x_n,p)=1$. This stronger condition helps us in our later calculations because it results in taking differences of Gauss sums which often cancel each other, as seen in subsection 3.3.1 below. To see that $\mathcal{B}(Q)$ is non-empty, we note that if $p\ge 5$, then each of the sets 
$$
S_1:=\left\{\lambda_1x_1^2+\cdots +\lambda_{n-2}x_{n-2}^2 \in \mathbb{F}_p: (x_1,...,x_{n-2})\in \left(\mathbb{F}_p^{\times}\right)^{n-2}\right\}
$$
and 
$$
S_2:=\left\{\lambda_{n+1}-\lambda_{n-1}x_{n-1}^2-\lambda_{n}x_{n}^2 \in \mathbb{F}_p: (x_{n-1},x_{n})\in \left(\mathbb{F}_p^{\times}\right)^{2}\right\}
$$
contains at least $(p+1)/2$ elements, and hence their intersection is non-empty, which implies that $\mathcal{B}(Q)$ is non-empty. Indeed, the set $\left\{\lambda_ix_i^2: x_i\in \mathbb{F}_p^{\times}\right\}$ equals either the set of quadratic residues or that of quadratic non-residues modulo $p$, which both have a cardinality of $(p-1)/2$. By the Cauchy-Davenport theorem, the sum of two such sets has at least $p-2$ elements, which is greater than or equal to $(p+1)/2$ if $p\ge 5$.  
\end{remark}
The proof of Theorem \ref{thm1} starts with detecting the congruence condition $Q(x_1,...,x_n)\equiv 0 \bmod{q}$ using additive characters, followed by an application of the Poisson summation formula. We isolate the main term, which arises from the zero frequency. The error term contains quadratic Gauss sums. After their evaluation, we are led to Kloosterman and Sali\'e sums, which we also evaluate explicitly. This leads us to  expressions which contain square roots modulo prime powers. They essentially take the form
$$
\sum\limits_{\substack{0<k\le K\\ (k,p)=1}} \tau_n(k)\cos\left(\frac{2\pi \sqrt{k\Lambda}}{p^s}\right),
$$    
where $\Lambda$ is coprime to $p$, $\sqrt{k\Lambda}$ denotes one of the two square roots of $k\Lambda$ modulo $p^s$, if existent (otherwise, the summand is omitted), and $\tau_n(k)$ denotes a weighted number of integral solutions $(x_1,...,x_n)$ to a Diophantine equation of the form
$$
\tilde{Q}(x_1,...,x_n)=k,
$$ 
$\tilde{Q}$ being a dual quadratic form. We approximate $\tau_n(k)$ using a variant of the circle method and are then led to sums of terms of the form 
$\cos\left(2\pi \sqrt{k\Lambda}/p^s\right)$ with $k$ running over residue classes. We establish that these sums are bounded by essentially the square root of the modulus $p^s$ using a Hensel-type argument and a standard bound for linear exponential sums. This allows us to bound the error term satisfactorily. 

The proof of Theorem \ref{thm1'} uses the same basic idea as the proof of Theorem \ref{mainresult2BaHa} in \cite{BaHa} but is much simpler as a result of the increased number of variables. It begins with bounding the error term by a dual count of solutions of quadratic congruences and then makes use of the Cauchy-Schwarz inequality. This leads us to counting solutions of linear congruences, which turns out easy if the number of variables is greater than or equal to 4.   

We point out that a possible alternative approach to the above problems consists of writing the congruence $Q({\bf x})\equiv 0\bmod{q}$ as an equation $Q({\bf x})=lq$, detecting weighted solutions of this equation directly using the circle method and summing up over $l$ in the relevant range. However, isolating main terms in the simple forms as they appear in our Theorems \ref{thm1} and \ref{thm1'} above then requires much more work. In contrast, the approach in this paper yields these main terms in a natural and easy way, and for the error terms we obtain essentially optimal bounds.  \\ \\
{\bf Acknowledgements.} The authors would like to thank the anonymous referees for their valuable comments. They would also like to thank the Ramakrishna Mission Vivekananda Educational and Research Institute for providing excellent working conditions. The second-named author would like to thank CSIR, Govt. of India for financial support in the form of a Junior Research Fellowship under file number 09/934(0015)/2019-EMR-I. \\ \\
{\bf Data availability statement.} No data are associated to this article.\\ \\
{\bf Conflict of interest statement.} The authors declare no conflict of interest regarding this article. 

\section{Preliminaries} \label{prelim}
\subsection{Notations}
Throughout this article, we use the following notations.
\begin{itemize}
\item We write
$$
e(z)=e^{2\pi i z}
$$
if $z\in \mathbb{R}$. 
\item For $q\in \mathbb{N}$ and $a\in \mathbb{Z}$, we write
$$
e_q(a)=e\left(\frac{a}{q}\right).
$$
This is an additive character modulo $q$. 
\item If $c$ is odd, then $\left(\frac{a}{c}\right)$ denotes the Jacobi symbol.
\item If $c$ is odd, then we set 
$$\epsilon_c:=\begin{cases}
        1,&\text{if }c\equiv1\bmod4,\\
        i,&\text{if }c\equiv3\bmod4.
    \end{cases}
$$ 
\item For $a,b\in \mathbb{Z}$ and $c\in \mathbb{N}$, we denote the generalized quadratic Gauss sum by
$$
G(a,b,c):=\sum_{n=0}^{c-1}e_c(an^2+bn).
$$ 
\item For $a,b\in \mathbb{Z}$ and $c\in \mathbb{N}$, we denote the Kloosterman sum by 
$$
K_0(a,b,c):=\sum\limits_{\substack{n=0\\ (n,c)=1}}^{c-1} e_c\left(a\overline{n}+bn\right)
$$
where for $n$ coprime to $c$, $\overline{n}$ denotes a multiplicative inverse of $n$ modulo $c$, i.e., $n\overline{n} \equiv 1 \bmod{c}$. If $c$ is odd, we denote the Sali\'e sum by
$$
K_1(a,b,c):=\sum\limits_{\substack{n=0\\ (n,c)=1}}^{c-1} \left(\frac{n}{c}\right) e_c\left(a\overline{n}+bn\right).
$$
\item We abbreviate $n$-dimensional vectors $(x_1,...,x_n)$, $(y_1,...,y_n)$, $(k_1,...,k_n)$ etc. by the corresponding bold letters ${\bf x}$, ${\bf y}$, ${\bf k}$ etc.. 
\item Following usual convention, $\varepsilon$ stands for an arbitrarily small positive number which may change from line to line. All implied $O$-constants are allowed to depend on $\varepsilon$.  
\end{itemize}

\subsection{Poisson summation formula}
We will use the following variant of the Poisson summation formula for residue classes. 

\begin{proposition} \label{Poisson}
Let $q\in \mathbb{N}$, $a\in \mathbb{Z}$, $N>0$ and $\Phi: \mathbb{R} \rightarrow \mathbb{C}$  be a Schwartz class function. Then 
$$
\sum\limits_{\substack{m\in \mathbb{Z}\\ m\equiv a \bmod{q}}} \Phi\left(\frac{m}{N}\right) =\frac{N}{q}\sum\limits_{n\in \mathbb{Z}} \hat{\Phi}\left(\frac{nN}{q}\right) e\left(\frac{na}{q}\right),
$$ 
where $\hat\Phi :\mathbb{R}\rightarrow \mathbb{C}$ is the Fourier transform of $\Phi$, defined by
$$
\hat{\Phi}(y):=\int\limits_{\mathbb{R}} \Phi(x)e(-xy){\rm d}x.
$$
\end{proposition}

\begin{proof} This follows using a linear change of variables from the ordinary Poisson summation formula which asserts that 
$$
\sum\limits_{m\in \mathbb{Z}} \Psi(m)= \sum\limits_{n\in \mathbb{Z}} \hat{\Psi}(n)
$$
for all Schwartz class functions $\Psi : \mathbb{R} \rightarrow \mathbb{C}$. 
For background on the Poisson summation formula see \cite{StSh}. 
\end{proof}

\subsection{Results on Gauss sums}
The following proposition gives well-known properties of generalized quadratic Gauss sums.

\begin{proposition} \label{Gaussprop} Let $a,b\in \mathbb{Z}$ and $c\in \mathbb{N}$. Then we have the following.\medskip\\
{\rm (i)}  If $(a,c)|b$, then 
$$
G(a,b,c)=(a,c) \cdot G\left(\frac{a}{(a,c)},\frac{b}{(a,c)},\frac{c}{(a,c)}\right). 
$$
{\rm (ii)}  If $(a,c)\nmid b$, then $G(a,b,c)=0$.\medskip\\
{\rm (iii)} If $(a,c)=1$ and $c$ is odd,  then 
$$
G(a,b,c)=e\left(-\frac{\overline{4a}b^2}{c}\right)G(a,0,c), 
$$
where for $x\in \mathbb{Z}$ coprime to $c$, $\overline{x}$ denotes a multiplicative inverse of $x$ modulo $c$, i.e., $x\overline{x} \equiv 1 \bmod{c}$. \medskip\\ 
{\rm (iv)} If $(a,c)=1$ and $c$ is odd, then 
$$
G(a,0,c)=\left(\frac{a}{c}\right)\epsilon_cc^\frac{1}{2}.
$$ 
{\rm (v)} If $(a,c)=1$, then 
$$
|G(a,0,c)|\le 2\sqrt{c}.
$$
\end{proposition}
\begin{proof}
(i) Suppose that $(a,c)=d$ and $d|b$. Set $\tilde{a}:=a/d$, $\tilde{b}=b/d$ and $\tilde{c}=c/d$. Then we have  
$$
G(a,b,c)=\sum_{n=0}^{c-1}e_c(an^2+bn)=d\sum\limits_{n=0}^{\tilde{c}-1} e_{\tilde{c}}(\tilde{a}n^2+\tilde{b}n)=dG(\tilde{a},\tilde{b},\tilde{c}),
$$ 
establishing part (i). \medskip\\
(ii) Suppose that $(a,c)=d$ and $d\nmid b$. Set $\tilde{a}:=a/d$ and $\tilde{c}:=c/d$. Then we have 
$$
G(a,b,c)=\sum_{n=0}^{c-1}e_c(an^2+bn)=\sum\limits_{n=0}^{c-1} e_{\tilde{c}}\left(\tilde{a}n^2\right)e_c\left(bn\right)=\sum\limits_{n=0}^{\tilde{c}-1} e_{\tilde{c}}\left(\tilde{a}n^2\right)\sum\limits_{\substack{m=0\\ m\equiv n\bmod{\tilde{c}}}}^{c-1} e_c(bm)=0
$$
since 
$$
\sum\limits_{\substack{m=0\\ m\equiv n\bmod{\tilde{c}}}}^{c-1} e_c(bm)=\sum\limits_{k=0}^{d-1} e_c(b(k\tilde{c}+n))=e_c(bn)\sum\limits_{k=0}^{d-1} e_d(bk)=0,
$$
establishing part (ii).\medskip\\
(iii) Suppose that $(a,c)=1$ and $c$ is odd. Then quadratic completion gives
$$
G(a,b,c)=\sum_{n=0}^{c-1}e_c(an^2+bn)=e\left(-\frac{\overline{4}ab^2}{c}\right)\sum_{n=0}^{c-1} e_c\left(a(n+\overline{2a}b)^2\right)=e\left(-\frac{\overline{4}ab^2}{c}\right) \sum\limits_{n=0}^{c-1} e_c(an^2)=e\left(-\frac{\overline{4}ab^2}{c}\right) G(a,0,c),
$$
establishing part (iii). \medskip\\
Parts (iv) and (v) are consequences of \cite[Lemmas 7.12-15.]{GrKo}. 
\end{proof}
 
\subsection{Results on Kloosterman and Sali\'e sums}
We shall use the following result to show the vanishing of Kloosterman and Sali\'e sums in certain instances.
\begin{proposition}\label{cochrane}
    Let $p\ge 3$ be a prime, $n\geq 2$ be a natural number and $f = F_1/F_2$ be a rational function where $F_1, F_2 \in\mathbb{Z}[x]$. For a polynomial $G$ over $\mathbb{Z}$, let $ord_p(G)$ be the largest power of $p$ dividing all of the coefficients of $G$, and for a rational function $g = G_1/G_2$ with $G_1$ and $G_2$ polynomials over $\mathbb{Z}$, let $ord_p(g) := ord_p(G_1)-ord_p(G_2)$. Set $$r := ord_p(f')$$ 
    and $$S_\alpha (f; p^n) :=\sum_{\substack{x=1\\x\equiv \alpha\bmod p}}^{p^n} e_{p^n} (f(x)),$$ 
    where $\alpha\in\mathbb{Z}$. If $r\leq n-2$, $(F_2(\alpha), p) = 1$ and $p^{-r}f'(\alpha)\not\equiv 0 \bmod p $, then we have $S_\alpha (f; p^n)=0.$
\end{proposition}
\begin{proof}
    This is a part of \cite[Theorem 3.1(iii)]{Coch}.
\end{proof}
The following proposition gives evaluations of Kloosterman and Sali\'e sums.
\begin{proposition} \label{Kloosterman} Let $a,b\in \mathbb{Z}$, $c\in \mathbb{N}$ odd, $p$ be an odd prime and $n\in \mathbb{N}$. \medskip\\
{\rm (i)} If $c=p^n$ with $n\geq2$ and $p\nmid ab$, then 
$$ K_0(a,b,c)= 2\left(\frac{v}{c}\right)c^\frac{1}{2}\Re\left(\epsilon_ce_c(2v)\right), $$ 
provided that $v$ is a solution of the congruence $v^2\equiv ab \bmod{c}$. If this congruence is not solvable ${\rm (}$which happens if and only if $\left(\frac{ab}{p}\right)=-1${\rm )}, then 
$K_0(a,b,c)=0$. \medskip\\
{\rm (ii)} If $(ab,c)=1$, then 
$$ K_1(a,b,c)= \epsilon_c \left(\frac{b}{c}\right)c^\frac{1}{2}\sum_{v^2\equiv ab\bmod c}e_c(2v). $$
{\rm (iii)} If  $c=p^n$ with $n\geq2$, $p|a$ and $p\nmid b$, then 
$K_0(a,b,c)=0=K_1(a,b,c)$. 
\end{proposition}
\begin{proof}
See \cite[Page 322, Exercise 1]{IwKo} and \cite[Page 323, Lemma 12.4]{IwKo} for parts (i) and (ii). To prove part (iii) we  use the Proposition \ref{cochrane} as follows. The Kloosterman sum
$K_0(a,b,c)$ can be expressed as $$K_0(a,b,c)= \sum_{\alpha=1}^{p-1} S_\alpha (f;p^n), $$ where 
$$
f(x)=\frac{a}{x}+bx=\frac{a+bx^2}{x}.
$$ 
If $p\nmid b$, then  $r=ord_p(f)=0$ and $f'(\alpha)\equiv 0\bmod p$ if and only if $\alpha^2 \equiv a\overline{b} \bmod p $. Now as $p|a$, we have $a\overline{b}\equiv 0\bmod p $ and so $\alpha^2\equiv0\bmod p$. But $\alpha$ varies from 1 to $p-1$, so $\alpha^2\equiv0\bmod p$ cannot hold. Hence if $1\leq \alpha\leq p-1$, then $p^{-r}f'(\alpha)\not\equiv 0 \bmod p $ and therefore $S_\alpha(f;p^n)=0$, resulting in $K_0(a,b,c)=0$.

The Sali\'e sum $K_1(a,b,c)$ can be written as $$K_1(a,b,c)=\sum_{\substack{x=1\\(x,c)=1}}^{c-1} e_c(a\overline{x^2}+bx^2)- \sum_{\substack{x=1\\(x,c)=1}}^{c-1} e_c(a\overline{x}+bx). $$
By a similar argument as before we can show that both the exponential sums on the right-hand side of the above equation are zero when $c=p^n$, $p|a$ and $p\nmid b$. This completes the proof of part (iii).
\end{proof}

\subsection{Weighted representations by quadratic forms} We shall use the following result on weighted numbers of representations by quadratic forms, which is a variant of \cite[Theorem 1.1]{Jon}.

\begin{proposition} \label{modHeath}
Fix an odd prime $p$ and a homogeneous quadratic form  $F^{(0)}(x_1,...,x_n)\in \mathbb{Z}[x_1,...,x_n]$ in $n\ge 4$ variables which is non-singular modulo $p$ and positive definite over $\mathbb{R}$. Fix a bump function $\Omega:\mathbb{R} \rightarrow \mathbb{R}_{\ge 0}$. Suppose that $k\in \mathbb{N}$ and $P\ge 1$. For ${\bf x}=(x_1,...,x_n)\in \mathbb{R}^n$ set 
$$
w({\bf x}):=\prod_{i=1}^n\Omega(x_i)
$$
and 
$$
w_P({\bf x}):=w\left(P^{-1}{\bf x}\right).
$$
Let 
$$
\sigma_{\infty,P}(k):=\lim\limits_{\varepsilon\rightarrow 0^+} \frac{1}{2\epsilon} \int\limits_{|F^{(0)}({\bf x})-k/P^2|< \epsilon} w({\bf x}){\rm d}{\bf x}
$$
be the singular integral and 
\begin{equation*}
    \sigma(k):=\sum_{q=1}^\infty a_q(k) 
\end{equation*}
be the singular series, where 
\begin{equation}\label{series}
    a_q(k):=\frac{1}{(pq)^n}\sum_{\substack{a\bmod q\\(a,q)=1}}\sum\limits_{\substack{{\bf x}\bmod pq\\ (x_1\cdots x_n,p)=1}}e_q\left(a(F^{(0)}({\bf x})-k) \right).
\end{equation}
Then, 
\begin{equation} \label{heathbrownasymp}
\sum_{\substack{{\bf x}\in\mathbb{Z}^n\\ F^{(0)}({\bf x})=k\\ (x_1\cdots x_n,p)=1}}w_P({\bf x})=\sigma_{\infty,P}(k)\sigma(k)P^{n-2}+O\left(\left(k^{n/2-1}P^{(3-n)/2}+P^{(n-1)/2}\right)(kP)^{\varepsilon}\right).
\end{equation}
\end{proposition}
\begin{proof}
Our proposition above is a slightly modified version of \cite[Theorem 1.1]{Jon}. These two results differ in the following points.\medskip\\
(i)  In \cite[Theorem 1.1]{Jon}, the precise dependence of the error term on the quadratic form was worked out. Here we assume that this form is fixed, and hence the error term simplifies. \medskip\\
(ii) There is an additional summation condition $(x_1\cdots x_n,p)=1$ on the left-hand side of \eqref{heathbrownasymp}. This reflects in the singular series, where we also have an extra summation condition of the same form in the definition of the summands $a_q(k)$. We note that here the modulus for the sum over ${\bf x}$ is extended from $q$ to $pq$, which is compensated by an extra factor of $p^{-n}$.\medskip\\
Proposition \ref{modHeath} can be established along the same lines as \cite[Theorem 1.1]{Jon}
using the Kloosterman refinement of the circle method, following the treatment in \cite[section 20.4]{IwKo}. 
\end{proof}


Below we record some properties of the function $\sigma_{\infty,P}(k)$ which will be needed later on.    

\begin{lemma} \label{Heathrem} Under the conditions of Proposition \ref{modHeath}, we have the following.  \medskip\\
{\rm (i)} There is a positive constant $M$ only depending on $w$ such that $|\sigma_{\infty,P}(t)|\le M$ if $t\ge 1$ and $P\ge 1$.\medskip\\
{\rm (ii)} Uniformly in $t\ge 1$ and $P\ge 1$, we have the bound $\frac{d}{dt} \sigma_{\infty,P}(t)\ll t^{-1}$.    
\end{lemma}

\begin{proof}
{\rm (i)} For any $y\ge 0$, we have 
$$
\sup\limits_{\substack{{\bf x}\in \mathbb{R}^n\\ F^{(0)}({\bf x})=y}} |w({\bf x})| \ll \left(1+y\right)^{-A},
$$
where $A>0$ is arbitrary, the implied constant only depending on $F^{(0)}$, $w$ and $A$. This is because
$w({\bf x})$ is a product of bump functions and $F^{(0)}({\bf x})$ is positive definite. It follows that 
$$
\sup\limits_{\substack{{\bf x}\in \mathbb{R}^n\\ \left|F^{(0)}({\bf x})-y\right|<\epsilon}} |w({\bf x})| \ll \left(1+y\right)^{-A}
$$
if $0<\epsilon\le y/2$. From this, we deduce the estimate 
\begin{equation*}
\lim\limits_{\varepsilon\rightarrow 0^+} \frac{1}{2\epsilon} \int\limits_{|F^{(0)}({\bf x})-y|< \epsilon} w({\bf x}){\rm d}{\bf x} \ll \left(1+y\right)^{-A}\lim\limits_{\varepsilon\rightarrow 0^+} \frac{1}{2\epsilon} \int\limits_{|F^{(0)}({\bf x})-y|< \epsilon} {\rm d}{\bf x}.
\end{equation*}
The limit on the right-hand side equals the $(n-1)$-dimensional volume of the hypersurface
$$
\mathcal{S}=\{{\bf x}\in \mathbb{R}^n : F^{(0)}({\bf x})=y\}.
$$
Since $F^{(0)}({\bf x})$ is a positive definite quadratic form, this volume is finite and proportional to $y^{(n-1)/2}$. Now taking $A:=(n-1)/2$, it follows that 
$$
\lim\limits_{\varepsilon\rightarrow 0^+} \frac{1}{2\epsilon} \int\limits_{|F^{(0)}({\bf x})-y|< \epsilon} w({\bf x}){\rm d}{\bf x} = O(1)
$$ 
and hence $\sigma_{\infty,P}(t)=O(1)$, establishing the claim in part (i).\medskip\\
(ii) We transform the derivative with respect to $t$ via the following chain of equations using linear changes of variables and the Leibniz rule on the differentiation of integrals. 
\begin{equation} \label{initr}
\begin{split}
 \frac{d}{dt} \sigma_{\infty,P}(t) = & \frac{d}{dt}\left(\lim\limits_{\varepsilon\rightarrow 0^+} \frac{1}{2\epsilon} \int\limits_{|F^{(0)}({\bf x})-t/P^2|< \epsilon} w({\bf x}){\rm d}{\bf x}\right)\\
= &  \frac{d}{dt}\left(\lim\limits_{\varepsilon\rightarrow 0^+} \frac{1}{2\epsilon} \int\limits_{|F^{(0)}(t^{-1/2}{\bf x})-1/P^2|< \epsilon/t} w({\bf x}){\rm d}{\bf x}\right)\\
= & \frac{d}{dt}\left(t^{n/2}\lim\limits_{\varepsilon\rightarrow 0^+} \frac{1}{2\epsilon} \int\limits_{|F^{(0)}({\bf x})-1/P^2|< \epsilon/t} w(t^{1/2}{\bf x}){\rm d}{\bf x}\right)\\ 
= & \frac{d}{dt}\left(t^{n/2-1}\lim\limits_{\varepsilon\rightarrow 0^+} \frac{1}{2\epsilon} \int\limits_{|F^{(0)}({\bf x})-1/P^2|< \epsilon} w(t^{1/2}{\bf x}){\rm d}{\bf x}\right)\\ 
= &  f_1(t)+f_2(t),
\end{split}
\end{equation}
where 
$$
f_1(t):=\left(\frac{n}{2}-1\right) \cdot t^{n/2-2}\lim\limits_{\varepsilon\rightarrow 0^+} \frac{1}{2\epsilon} \int\limits_{|F^{(0)}({\bf x})-1/P^2|< \epsilon} w(t^{1/2}{\bf x}){\rm d}{\bf x}$$
and 
$$
f_2(t):= t^{n/2-1}\lim\limits_{\varepsilon\rightarrow 0^+} \frac{1}{2\epsilon} \int\limits_{|F^{(0)}({\bf x})-1/P^2|< \epsilon} \frac{d}{dt} w(t^{1/2}{\bf x}){\rm d}{\bf x}.
$$
Reversing the steps carried out in \eqref{initr}, we deduce that
\begin{equation} \label{f1bound}
\begin{split}
f_1(t)\ll & t^{n/2-2}\lim\limits_{\varepsilon\rightarrow 0^+} \frac{1}{2\epsilon} \int\limits_{|F^{(0)}({\bf x})-1/P^2|< \epsilon} w(t^{1/2}{\bf x}){\rm d}{\bf x}\\
= & t^{n/2-1}\lim\limits_{\varepsilon\rightarrow 0^+} \frac{1}{2\epsilon} \int\limits_{|F^{(0)}({\bf x})-1/P^2|< \epsilon/t} w(t^{1/2}{\bf x}){\rm d}{\bf x}\\
= & t^{-1}\lim\limits_{\varepsilon\rightarrow 0^+} \frac{1}{2\epsilon} \int\limits_{|F^{(0)}(t^{-1/2}{\bf x})-1/P^2|< \epsilon/t} w({\bf x}){\rm d}{\bf x}\\ 
= & t^{-1}\lim\limits_{\varepsilon\rightarrow 0^+} \frac{1}{2\epsilon} \int\limits_{|F^{(0)}({\bf x})-t/P^2|< \epsilon} w({\bf x}){\rm d}{\bf x}\\
= & t^{-1} \sigma_{\infty,P}(t) \\
\ll & t^{-1},
\end{split}
\end{equation}
using part (i) for the last line.  Calculating $\frac{d}{dt}w\left(t^{1/2}{\bf x}\right)$ using the definition of $w$, we obtain
$$
f_2(t)=t^{n/2-1}\lim\limits_{\varepsilon\rightarrow 0^+} \frac{1}{2\epsilon} \int\limits_{|F^{(0)}({\bf x})-1/P^2|< \epsilon} (2t)^{-1}\left(\frac{t^{1/2}x_1\Omega'(t^{1/2}x_1)}{\Omega(t^{1/2}x_1)}+\cdots +\frac{t^{1/2}x_n\Omega'(t^{1/2}x_n)}{\Omega(t^{1/2}x_n)}\right)w(t^{1/2}{\bf x}) {\rm d}{\bf x}.
$$
Again applying similar transformations as in \eqref{initr} backwards, we obtain
$$
f_2(t)= (2t)^{-1}\lim\limits_{\varepsilon\rightarrow 0^+} \frac{1}{2\epsilon} \int\limits_{|F^{(0)}({\bf x})-t/P^2|< \epsilon} \left(\frac{x_1\Omega'(x_1)}{\Omega(x_1)}+\cdots + \frac{x_n\Omega'(x_n)}{\Omega(x_n)}\right)w({\bf x}){\rm d}{\bf x}. 
$$
Now, by similar arguments as in part (i), the limit above is $O(1)$ and hence
\begin{equation} \label{f2bound}
f_2(t)\ll t^{-1}.
\end{equation}
Putting \eqref{initr}, \eqref{f1bound} and \eqref{f2bound} together,  the claim in part (ii) follows. 
\end{proof}

\section{Proof of Theorem \ref{thm1} (inhomogeneous congruences)}
\subsection{Application of Poisson summation}
Our quantity in question is 
$$
T=\sum_{\substack{{\bf x}\in\mathbb{Z}^n\\(x_1x_2\cdots x_n,p)=1\\Q({\bf x})\equiv 0\bmod p^m}} \prod_{i=1}^n\Phi\left(\frac{x_i}{N}\right).
$$
We detect the linear congruence using addititive characters and re-arrange the summations, obtaining
\begin{equation*}
\begin{split}
    T&=\frac{1}{p^m}\sum_{\substack{{\bf x}\in\mathbb{Z}^n\\(x_1x_2\cdots x_n,p)=1}} \left(\prod_{i=1}^n\Phi\left(\frac{x_i}{N}\right)\right)\left( \sum_{h=1}^{p^m}e_{p^m}(hQ({\bf x}))\right)\\
    &= \frac{1}{p^m}\sum_{h=1}^{p^m}e_{p^m}(-h\lambda_{n+1}) \prod_{i=1}^n\Bigg( \sum_{\substack{x_i\in\mathbb{Z}\\(x_i,p)=1}} \Phi\left(\frac{x_i}{N}\right) e_{p^m}\left(h\lambda_{i}x_i^2\right)\Bigg).
\end{split}
\end{equation*}
Dividing the $x_i$'s into residue classes modulo $p^m$, using the Poisson summation formula, Proposition \ref{Poisson}, and re-arranging summations, we deduce that 
\begin{equation*}
\begin{split}
    T&= \frac{1}{p^m}\sum_{h=1}^{p^m}e_{p^m}(-h\lambda_{n+1}) \prod_{i=1}^n\Bigg(
\sum\limits_{\substack{y_i=1\\ (y_i,p)=1}}^{p^m} e_{p^m}\left(h\lambda_{i}y_i^2\right)\sum_{\substack{x_i\in\mathbb{Z}\\ x_i\equiv y_i\bmod{p^m}}} \Phi\left(\frac{x_i}{N}\right)\Bigg)\\
&= \frac{N^n}{p^{m(n+1)}}\sum_{h=1}^{p^m}e_{p^m}(-h\lambda_{n+1}) \prod_{i=1}^n
\Bigg(\sum\limits_{\substack{y_i=1\\ (y_i,p)=1}}^{p^m} e_{p^m}\left(h\lambda_{i}y_i^2\right)\sum_{k_i\in \mathbb{Z}} \hat{\Phi}\left(\frac{k_iN}{p^m}\right)e_{p^m}(k_iy_i)\Bigg)\\
&= \frac{N^n}{p^{m(n+1)}}\sum_{\substack{{\bf k}\in\mathbb{Z}^n}}\left(\prod_{i=1}^n\hat{\Phi}\left(\frac{k_iN}{p^m}\right)\right)\Bigg(\sum_{h=1}^{p^m} e_{p^m}(-h\lambda_{n+1})\prod_{j=1}^{n}\sum_{\substack{y_j=1\\(y_i,p)=1}}^{p^m} e_{p^m}(h\lambda_{j}y_j^2+k_jy_j)\Bigg).
\end{split}
\end{equation*}
\subsection{Evaluation of the main term} The main term $T_0$ comes from the zero frequency, i.e. the contribution of ${\bf k}={\bf 0}$ to the last line above. We obtain
\begin{equation*}
\begin{split}
    T_0&= \frac{N^n}{p^{m(n+1)}}\cdot \hat{\Phi}(0)^n\cdot \sum_{h=1}^{p^m} e_{p^m}(-h\lambda_{n+1})\prod_{j=1}^{n}\sum_{\substack{y_j=1\\(y_i,p)=1}}^{p^m} e_{p^m}(h\lambda_{j}y_j^2)\\
    &= \frac{N^n}{p^{m(n+1)}}\cdot \hat{\Phi}(0)^n \cdot \sum_{\substack{y_1,...,y_n=1\\(y_1...y_n,p)=1}}^{p^m}\sum_{h=1}^{p^m}e_{p^m}(hQ(y_1,...,y_n)) \notag\\
    &= \frac{N^n}{p^m}\cdot\hat\Phi(0)^n\cdot \frac{\sharp\mathcal{B}_m}{p^{m(n-1)}},
\end{split}
\end{equation*}
where 
$$
\mathcal{B}_m:=\left\{{\bf y}\in\mathbb{Z}^n:~0\leq y_1,...,y_{n}\leq p^m-1,\ (y_1\cdots y_n,p)=1, \ Q({\bf y})\equiv 0 \bmod p^m\right\}.
$$
Now we note that the quantity $\sharp\mathcal{B}_m/p^{m(n-1)}$ is independent of $m$ as a consequence of Hensel's lemma. Therefore, we can replace this term by $\sharp\mathcal{B}_1/p^{n-1}$ and thus get
\begin{equation} \label{maintermev}
T_0= B_p(Q) \cdot\hat\Phi(0)^n\cdot \frac{N^n}{p^m},
\end{equation}
where $B_p(Q)$ is as given in Theorem \ref{thm1}. 

\subsection{Evaluation of the error term} 
Now we turn to the evaluation of our error term 
\begin{equation} \label{theerror}
U:=T-T_0=\frac{N^n}{p^{m(n+1)}}\sum_{\substack{{\bf k}\in\mathbb{Z}^n\\ {\bf k}\not={\bf 0}}}\Psi({\bf k})F({\bf k}), 
\end{equation}
where 
\begin{equation} \label{weight}
\Psi({\bf k}):=\prod_{i=1}^n\hat{\Phi}\left(\frac{k_iN}{p^m}\right)
\end{equation}
and 
\begin{equation} \label{Ffunction}
F({\bf k}):=\sum_{h=1}^{p^m} e_{p^m}(-h\lambda_{n+1})\prod_{j=1}^{n}\sum_{\substack{y_j=1\\(y_i,p)=1}}^{p^m} e_{p^m}(h\lambda_{j}y_j^2+k_jy_j).
\end{equation}

\subsubsection{Evaluation of Gauss sums}
Suppose that $(k_1,k_2,...,k_n,p^m)=p^r$. We first note that using the rapid decay of $\hat{\Phi}$ and ${\bf k}\not=0$, the contributions of $r=m-1$ and $r=m$ to the right-hand side of \eqref{theerror} are negligible and hence bounded by $O(1)$ if we assume that $N\geq p^{m\varepsilon}$ and $m$ tends to infinity. So we left with the contributions of $r\in \{0,...,m-2\}$. 

We first write the inner-most sum over $y_j$ on the right-hand side of \eqref{Ffunction} as  
\begin{equation*}
\begin{split}
    \sum_{\substack{y_j=1\\(y_j,p)=1}}^{p^m} e_{p^m}(h\lambda_{j}y_j^2+k_jy_j)=& \sum_{\substack{y_j=1}}^{p^m} e_{p^m}(h\lambda_{j}y_j^2+k_jy_j)-\sum_{\substack{y_j=1}}^{p^{m-1}} e_{p^{m-1}}(h\lambda_{j}py_j^2+k_jy_j)\\
    = & G(h\lambda_j,k_j,p^m)-G(h\lambda_jp,k_j,p^{m-1}),
\end{split}
\end{equation*}
where $G(a,b,c)$ is the quadratic Gauss sum defined in the section \ref{prelim}. 
Suppose $0\leq r\leq m-2$ and let $j\in\{1,...,n\}$ be such that $(k_j,p^m)=p^r$. We consider the following three cases.\\ \\
\textbf{Case 1: \boldmath{$(h,p^m)=p^s$} with \boldmath{$s>r$}.}
In this case, $p^s\nmid k_j$ and hence $G(h\lambda_j,k_j,p^m)=0 $ by Proposition \ref{Gaussprop}(ii). Also, 
$$
(h\lambda_jp,p^{m-1})= \begin{cases} p^{s+1} & \mbox{ if } s\leq m-2\\
 p^s & \mbox{ if } s\in \{m-1,m\}.
\end{cases} 
$$
In both cases, $(h\lambda_jp,p^{m-1})\nmid k_j $ and so $G(h\lambda_jp,k_j,p^{m-1})=0$.\\ \\
\textbf{Case 2: \boldmath{$(h,p^m)=p^s$} with \boldmath{$s=r$}.}
As $(h\lambda_jp,p^{m-1})= p^{r+1}\nmid k_j$, $G(h\lambda_j p,k_j,p^{m-1})=0$ by Proposition \ref{Gaussprop}(ii). Moreover, 
$$
G(h\lambda_j,k_j,p^m)=p^r\cdot G(h'\lambda_j,l_j,p^{m-r}) 
$$ 
with $h=p^rh'$ and $k_j=p^rl_j$ by Proposition \ref{Gaussprop}(i). Further, as $(h'\lambda_j,p^{m-r})=1$, we get 
\begin{align*}
  p^r\cdot G(h'\lambda_j,l_j,p^{m-r})&=p^r\cdot e\left(-\frac{\overline{4h'\lambda_j}l_j^2}{p^{m-r}}\right)G(h'\lambda_j,0,p^{m-r})  \\
  &=\left(\frac{h'\lambda_j}{p^{m-r}}\right)\epsilon_{p^{m-r}}p^{(m+r)/2}e\left(-\frac{\overline{4h'\lambda_j}l_j^2}{p^{m-r}}\right)
\end{align*}
by parts (iii) and (iv) of Proposition \ref{Gaussprop}. \\ \\
\textbf{Case 3: \boldmath{$(h,p^m)=p^s$} with \boldmath{$s<r$}.}
Proceeding similarly as in case 2, 
\begin{equation*}
\begin{split}
    G(h\lambda_j,k_j,p^m)&=p^sG(h'\lambda_j,l_j,p^{m-s})\\
    &= p^se\left(-\frac{\overline{4h'\lambda_j}l_j^2}{p^{m-s}}\right) G(h'\lambda_j,0,p^{m-s}),
\end{split}
\end{equation*}
where $h=p^sh'$ and $k_j=p^sl_j$. 
Now $(h\lambda_jp,p^{m-1})=p^{s+1}$ and also $s+1\leq r$, so $p^{s+1}| k_j$. Thus we get
\begin{align*}
    G(h\lambda_j p,k_j,p^{m-1})&=p^{s+1}G(h'\lambda_j,l_j/p,p^{m-s-2})\\
    &= p^{s+1}e\left(-\frac{\overline{4h'\lambda_j}(l_j/p)^2}{p^{m-s-2}}\right) G(h'\lambda_j,0,p^{m-s-2})\\
&= p^{s+1}e\left(-\frac{\overline{4h'\lambda_j}l_j^2}{p^{m-s}}\right) G(h'\lambda_j,0,p^{m-s-2}),
\end{align*}
again using Proposition \ref{Gaussprop}(iii). 
As $G(h'\lambda_j,0,p^{m-s})=p\cdot G(h'\lambda_j,0,p^{m-s-2})$ using Proposition \ref{Gaussprop}(iv), we deduce that
$$
G(h\lambda_j,k_j,p^m)-G(h\lambda_j p,k_j,p^{m-1})=0. 
$$

Observing the above three cases we get that 
\begin{equation}
    G(h\lambda_j,k_j,p^m)-G(h\lambda_j p,k_j,p^{m-1})=\begin{cases}
        \left(\frac{h'\lambda_j}{p^{m-r}}\right)\epsilon_{p^{m-r}}p^{(m+r)/2}e\left(-\frac{\overline{4h'\lambda_j}l_j^2}{p^{m-r}}\right)&\text{ if }(h,p^m)=p^r,\\
        0&\text{ otherwise},
    \end{cases}
\end{equation}
where $h=p^rh'$ and $(h',p)=1$.
It follows that if $(k_1,k_2,...,k_n,p^m)=p^r$ for some $0\leq r\leq m-2$, then 
\begin{equation*}
F({\bf k})=\sum_{\substack{h=1\\(h,p^m)=p^r}}^{p^m} e_{p^m}(-h\lambda_{n+1})\prod_{j=1}^{n}\left(G(h\lambda_j,k_j,p^m)-G(h\lambda_jp,k_j,p^{m-1})\right).
\end{equation*}

Suppose there is some $i\in \{1,...,n\}$ with $(k_i,p^m)=p^u>p^r$, then if $u\leq m-2$, we have already shown in Case 3 that $$G(h\lambda_i,k_i,p^m)=G(h\lambda_i p,k_i,p^{m-1}).$$
If $u=m-1,m$, similar calculations as in Case 3 give that 
\begin{align*}
    G(h\lambda_i,k_i,p^m)&=p^re\left(-\frac{\overline{4h'\lambda_i}l_i^{2}}{p^{m-r}} \right)G(h'\lambda_i,0,p^{m-r})\\
    &=p^{r+1}e\left(-\frac{\overline{4h'\lambda_i}l_i^{2}}{p^{m-r}} \right)G(h'\lambda_i,0,p^{m-r-2})\\
    &=G(h\lambda_i p,k_i,p^{m-1}).
\end{align*}
Hence, if $(k_1,k_2,...,k_n,p^m)=p^r$, then ${\bf k}$ gives a non-zero contribution to the error term $U$ if and only if $(k_j,p^m)=p^r$ for all $j\in \{1,...,n\}$. 

Summarizing the above considerations, we get
\begin{equation} \label{Usimp} 
\begin{split}
U=&  \frac{N^n}{p^{m(n+1)}}\sum_{r=0}^{m-2} \sum_{\substack{{\bf k}\in\mathbb{Z}^n\\(k_j,p^m)=p^r,\forall j}} \Psi({\bf k})F({\bf k}) +O(1)\\
 = & \frac{N^n}{p^{m(n+1)}}\sum_{r=0}^{m-2}\sum_{\substack{{\bf l}\in\mathbb{Z}^n\\(l_j,p)=1,\forall j}} \Psi(p^r{\bf l})F(p^r{\bf l})+O(1),
\end{split}
\end{equation}
and 
\begin{equation} \label{Fsimp}
F(p^r{\bf l})=\epsilon_{p^{m-r}}^n p^{n(m+r)/2} \left(\frac{\lambda_1\lambda_2\cdots \lambda_n}{p^{m-r}}\right) K_n(-A,-\lambda_{n+1},p^{m-r}),
\end{equation}
where 
$$
K_n(-A,-\lambda_{n+1},p^{m-r}):=\sum_{\substack{h=1\\ (h,p)=1}}^{p^{m-r}}\left(\frac{h}{p^{m-r}}\right)^ne_{p^{m-r}}(-A\overline{h}-\lambda_{n+1}h)
$$
and 
$$
A\equiv \overline{4}\left(\sum_{j=1}^n\overline{\lambda_j}l_j^2\right)\bmod p^{m-r}.
$$

\subsubsection{Evaluation of Kloosterman and Sali\'e sums}
The sum on the right-hand of \eqref{Fsimp} is a Kloosterman or Sali\'e sum, depending on whether $n$ is even or odd. We note that 
$$
K_n(-A,-\lambda_{n+1},p^{m-r})=K_{\tilde{n}} (A,\lambda_{n+1},p^{m-r}),
$$
where $\tilde{n}\in \{0,1\}$ with $\tilde{n}\equiv n\bmod{2}$. Recalling that $r\le m-2$, $(\lambda_{n+1},p)=1$ and $(l_1\cdots l_n,p)=1$, we deduce from parts (i) and (ii) of Proposition \ref{Kloosterman} that
\begin{equation} \label{KloSaEv}
    K_{\tilde{n}}(A,\lambda_{n+1},p^{m-r})
    = \begin{cases} 2\epsilon_{p^{m-r}}\left(\frac{\lambda_{n+1}}{p^{m-r}}\right) p^{(m-r)/2}\Re\left(e_{p^{m-r}}(2v)\right)
& \mbox{ if } n\equiv 1\bmod{2}\\ \\
2\left(\frac{v}{p^{m-r}}\right)p^{(m-r)/2} \Re\left(\epsilon_{p^{m-r}}e_{p^{m-r}}(2v)  \right) & \mbox{ if } n\equiv 0\bmod{2} \end{cases}
\end{equation}
if $(A,p)=1$ and $v$ is a solution of the congruence $v^2\equiv A\lambda_{n+1}\bmod p^{m-r}$. If no such solution exists or $(A,p)\not=1$, then $K_{\tilde{n}}(A,\lambda_{n+1},p^{m-r})=0$, using part (iii) of Proposition \ref{Kloosterman}. \\ \\
{\bf Case 1: \boldmath{$n$} odd.} In this case, it follows that
\begin{equation}\label{eqodd}
    U= \frac{N^n}{p^{m(n+1)/2}}\sum_{r=0}^{m-2}\epsilon_{p^{m-r}}^{n+1} p^{(n-1)r/2}\left(\frac{\Delta}{p^{m-r}}\right)\sum_{\substack{{\bf l}\in\mathbb{Z}^n\\(l_1\cdots l_n,p)=1\\ (A,p)=1}} 2\Psi(p^r{\bf l})\Re\left(e_{p^{m-r}}(2v)\right) +O(1),
\end{equation}
where $\Delta=\lambda_1\lambda_2...\lambda_n\lambda_{n+1}$. We can write the above congruence $v^2\equiv A\lambda_{n+1}\bmod p^{m-r}$ in the form 
$$ 
 (2v)^2\equiv \tilde{Q}({\bf l})\Lambda \bmod p^{m-r},
$$
where 
$$
\tilde{Q}({\bf l}):=\sum_{j=1}^n\Delta_jl_j^2 \quad \mbox{and} \quad \Delta_{n+1}\Lambda \equiv 1 \bmod{p^{m}},
$$
with $\Delta_j:=\Delta/\lambda_j $ for $j=1,2,...,n+1$. We note that $(A,p)=1$ if and only if $(\tilde{Q}({\bf l}),p)=1$. For $k\in \mathbb{Z}$, we define 
\begin{equation} \label{taudef}
    \tau_n(k):=\sum_{\substack{{\bf l}\in\mathbb{Z}^n\\(l_1\cdots l_n,p)=1\\ \tilde{Q}({\bf l})=k}}\Psi(p^r{\bf l}).
    \end{equation}
We also note that $\tilde{Q}({\rm l})$ is positive definite if $\lambda_{n+1}>0$ and negative definite if $\lambda_{n+1}<0$ since $\lambda_1,...,\lambda_n$ are supposed to be positive in Theorem \ref{thm1}. In what follows, we treat only the case when $\lambda_{n+1}>0$, the other case being similar. In this case, $\tilde{Q}({\bf l})$ only represents non-negative $k$'s.

Now, re-arranging summations, we obtain
\begin{equation*}
\begin{split}
    U = \frac{N^n}{p^{m(n+1)/2}}\sum_{r=0}^{m-2}\epsilon_{p^{m-r}}^{n+1} p^{(n-1)r/2}\left(\frac{\Delta}{p^{m-r}}\right) \sum_{\substack{k\in \mathbb{N}\\ (k,p)=1}} 2\tau_n(k)\Re\left(e_{p^{m-r}}(2v)\right) +O(1),
\end{split}
\end{equation*}
where $v$ is a solution of the congruence $(2v)^2\equiv k\Lambda \bmod{p^{m-r}}$, if existent. Using Hensel's lemma, we have
$$
2\Re\left(e_{p^{m-r}}(2v)\right)=\sum\limits_{u^2\equiv k\Lambda\bmod{p^{m-r}}} e_{p^{m-r}}(u).
$$ 
We also observe that the $k$-sum above can be truncated at 
\begin{equation} \label{Kdef}
K:=nM^2_r\Delta_{\max},
\end{equation}
where 
\begin{equation} \label{Mrdef}
M_r:=\frac{p^{m-r+\varepsilon m}}{N} \quad \mbox{and} \quad \Delta_{\max}:= \max \{ |\Delta_j|: j=1,2,...,n \},
\end{equation}
at the cost of a negligible error. (Here it becomes essential that $\lambda_1,...,\lambda_n$ are fixed.) Hence, $U$ takes the form 
\begin{equation}\label{odderror}
    U = \frac{N^n}{p^{m(n+1)/2}}\sum_{r=0}^{m-2}\epsilon_{p^{m-r}}^{n+1}p^{(n-1)r/2}\left(\frac{\Delta}{p^{m-r}}\right) \sum_{\substack{0<k\leq K\\ (k,p)=1}}\tau_n(k) \sum\limits_{u^2\equiv k\Lambda\bmod{p^{m-r}}} e_{p^{m-r}}(u)+O(1)
\end{equation}
if $n$ is odd. \\ \\
{\bf Case 2(i): \boldmath{$n$} even, \boldmath{$r\le m-2$} and \boldmath{$\epsilon_{p^{m-r}}=1$}.} In this case, we necessarily have $p\equiv 1 \bmod{4}$ or $p\equiv3 \bmod{4}$ and $m-r$ even, and hence
$$
\left(\frac{a}{p^{m-r}}\right)=\left(\frac{-a}{p^{m-r}}\right)
$$ 
for all $a\in \mathbb{Z}$. Therefore, using \eqref{KloSaEv} for $n\equiv 0\bmod{2}$, we obtain
\begin{equation*}
\begin{split}
     K_{\tilde{n}}(A,\lambda_{n+1},p^{m-r})=& 2p^{(m-r)/2}\left(\frac{v}{p^{m-r}}\right) \Re\left(e_{p^{m-r}}(2v)  \right)\\ 
    = & p^{(m-r)/2}\left(\frac{v}{p^{m-r}}\right)\left(e_{p^{m-r}}(2v)+e_{p^{m-r}}(-2v)\right)\\
    = & p^{(m-r)/2}\left(\frac{2}{p^{m-r}}\right) \sum\limits_{u^2\equiv k\Lambda\bmod{p^{m-r}}} \left(\frac{u}{p^{m-r}}\right)e_{p^{m-r}}(u).
\end{split}
\end{equation*}
Now, similarly as in Case 1, we deduce that the contribution of $r$ to $U$ equals
$$
 \frac{N^n}{p^{m(n+1)/2}}\cdot \sum_{r=0}^{m-2}p^{(n-1)r/2}\left(\frac{2\Delta_{n+1}}{p^{m-r}}\right) \sum_{\substack{0<k\leq K\\ (k,p)=1}}\tau_n(k) \sum\limits_{u^2\equiv k\Lambda\bmod{p^{m-r}}} \left(\frac{u}{p^{m-r}}\right)e_{p^{m-r}}(u)+O(1).
$$
{\bf Case 2(ii): \boldmath{$n$} even, \boldmath{$r\le m-2$} and \boldmath{$\epsilon_{p^{m-r}}=i$}.}
In this case, we necessarily have $p\equiv3 \bmod{4}$ and $m-r$ odd and hence
$$
\left(\frac{a}{p^{m-r}}\right)=-\left(\frac{-a}{p^{m-r}}\right)
$$ 
for all $a\in \mathbb{Z}$. Therefore, using \eqref{KloSaEv}, we obtain
\begin{equation*}
\begin{split}
     K_{\tilde{n}}(A,\lambda_{n+1},p^{m-r})=& 2p^{(m-r)/2}\left(\frac{v}{p^{m-r}}\right) \Re\left(ie_{p^{m-r}}(2v)  \right)\\ 
    = & p^{(m-r)/2}\left(\frac{v}{p^{m-r}}\right)\left(ie_{p^{m-r}}(2v)-ie_{p^{m-r}}(-2v)\right)\\
    = & ip^{(m-r)/2}\left(\left(\frac{v}{p^{m-r}}\right)e_{p^{m-r}}(2v)+\left(\frac{-v}{p^{m-r}}\right)e_{p^{m-r}}(-2v)\right)\\
    = & ip^{(m-r)/2}\left(\frac{2}{p^{m-r}}\right) \sum\limits_{u^2\equiv k\Lambda\bmod{p^{m-r}}} \left(\frac{u}{p^{m-r}}\right)e_{p^{m-r}}(u).
\end{split}
\end{equation*}
Again, along similar lines as above, we obtain the same contribution of $r$ to $U$ as in Case 2(ii), with an extra factor of $i$. Combining the results in Cases 2(i) and (ii), we deduce that
\begin{equation}\label{evenerror}
    U = \frac{N^n}{p^{m(n+1)/2}}\sum_{r=0}^{m-2}\epsilon_{p^{m-r}}^{n+1}p^{(n-1)r/2}\left(\frac{2\Delta_{n+1}}{p^{m-r}}\right)
\sum_{\substack{0<k\leq K\\ (k,p)=1}}\tau_n(k) \sum\limits_{u^2\equiv k\Lambda\bmod{p^{m-r}}} \left(\frac{u}{p^{m-r}}\right)e_{p^{m-r}}(u)+O(1)
\end{equation}
if $n$ is even. 

\subsubsection{Evaluation of sums involving square roots modulo prime powers}
To estimate the sums over $k$ in \eqref{odderror} and \eqref{evenerror}, we establish the following.   
\begin{proposition}\label{1stprop}
Let $p$ be a fixed odd prime, $a,\Lambda\in \mathbb{Z}$ with $(a\Lambda,p)=1$, $s\in \mathbb{N}$ with $s\ge 2$,  $b\in \mathbb{Z}$, $c\in \mathbb{N}$ and $0< K\le p^s$. Then
$$
\sum\limits_{\substack{ 0<k\le K\\ k\equiv b \bmod{c}\\ (k,p)=1}}\ \sum\limits_{\substack{u^2\equiv k\Lambda\bmod{p^s}\\ u\equiv a\bmod{p}}} e_{p^s}(u) \ll_p p^{s/2}\log p^s.
$$ 
\end{proposition}
\begin{proof}   
First suppose that $s=2t$ is even. Suppose that $w^2\equiv k\Lambda \bmod{p^t}$ and $w\equiv a\bmod{p}$. Then using Hensel's lemma, $w$ lifts uniquely to a solution $u$ of the congruence $u^2\equiv k\Lambda\bmod{p^s}$ satisfying $u\equiv a\bmod{p}$. More precisely, we have $u=w+hp^t$, where $h$ is unique modulo $p^t$ and satisfies
$$
    (w+hp^t)^2\equiv k\Lambda \bmod p^{2t}.
$$
This is equivalent to 
$$
    \frac{w^2-k\Lambda}{p^t}+2wh\equiv 0\bmod p^t,
$$
which in turn is equivalent to
$$
h \equiv\overline{2w}\cdot\frac{k\Lambda-w^2}{p^t}\bmod p^t,
$$
where $2w\cdot \overline{2w}\equiv 1 \bmod{p^t}$.
It follows that 
\begin{equation*}
\begin{split}
\bigg| \sum\limits_{\substack{0<k\leq K\\k\equiv b \bmod{c}\\ (k,p)=1}}\ \sum\limits_{\substack{u^2\equiv k\Lambda\bmod{p^s}\\ u\equiv a\bmod{p}}} e_{p^s}(u) \bigg| =& \bigg|
\sum\limits_{\substack{ w\bmod p^t\\w\equiv a \bmod{p}}}\ \sum\limits_{\substack{ k\equiv w^2\overline{\Lambda} \bmod p^t\\k\equiv b \bmod c\\0<k\leq K}} e_{p^{2t}}\left(w+\overline{2w}\cdot\frac{k\Lambda-w^2}{p^t}\cdot p^t  \right) \bigg| \\
   \le & \sum\limits_{\substack{ w\bmod p^t\\w\equiv a \bmod{p}}}\bigg| \sum_{\substack{d\in\mathbb{Z}\\  dp^t+f(w^2\overline{\Lambda})\equiv b \bmod c\\
0<dp^t+f(w^2\overline{\Lambda})\leq K}} 
e_{p^{t}}\left(\overline{2w}\Lambda d\right) \bigg|,
\end{split}
\end{equation*}
where for $z\in \mathbb{Z}$, $f(z)$ is defined by the relations 
$$
f(z)\equiv z \bmod{p^t} \quad \mbox{and} \quad 0\leq f(z) <p^t.
$$ 

Suppose that $g\in\mathbb{N}_0$ such that $c=p^g\tilde{c}$, where $(p,\tilde{c})=1$. If $g\le t$, then, writing $p^t\overline{p^t}\equiv 1\bmod{\tilde{c}}$, we get
\begin{equation*}
\begin{split}
& \sum\limits_{\substack{ w\bmod p^t\\w\equiv a \bmod{p}}}\bigg| \sum\limits_{\substack{d\in\mathbb{Z}\\  dp^t+f(w^2\overline{\Lambda})\equiv b \bmod c\\
0<dp^t+f(w^2\overline{\Lambda})\leq K}} 
e_{p^{t}}\left(\overline{2w}\Lambda d\right)\bigg|\\ = &
\sum\limits_{\substack{ w\bmod p^t\\w\equiv a \bmod{p}\\ f(w^2\overline{\Lambda})\equiv b\bmod{p^g}}}\bigg| \sum\limits_{\substack{d\in\mathbb{Z}\\  dp^t+f(w^2\overline{\Lambda})\equiv b \bmod \tilde{c}\\
0<dp^t+f(w^2\overline{\Lambda})\leq K}} 
e_{p^{t}}\left(\overline{2w}\Lambda d \right) \bigg|\\
= & \sum\limits_{\substack{ w\bmod p^t\\w\equiv a \bmod{p}\\ f(w^2\overline{\Lambda})\equiv b\bmod{p^g}}}\bigg| \sum\limits_{\substack{d\equiv \overline{p^t}(b-f(w^2\overline{\Lambda})) \bmod \tilde{c}\\
0<dp^t+f(w^2\overline{\Lambda})\leq K}} 
e_{p^{t}}\left(\overline{2w}\Lambda d \right)\bigg|\\
= & \sum\limits_{\substack{ w\bmod p^t\\w\equiv a \bmod{p}\\ f(w^2\overline{\Lambda})\equiv b\bmod{p^g}}}\bigg| \sum\limits_{\substack{\alpha\in \mathbb{Z}\\ 
0<(\alpha\tilde{c}+\overline{p^t}(b-f(w^2\overline{\Lambda})))p^t+f(w^2\overline{\Lambda})\leq K}} 
e_{p^{t}}\left(\overline{2w}\Lambda (\alpha\tilde{c}+\overline{p^t}(b-f(w^2\overline{\Lambda}))) \right) \bigg|\\
= & \sum\limits_{\substack{ w\bmod p^t\\w\equiv a \bmod{p}\\ f(w^2\overline{\Lambda})\equiv b\bmod{p^g}}}\bigg| \sum\limits_{\substack{\alpha\in \mathbb{Z}\\ 
\alpha\in I_w}} 
e_{p^{t}}\left(\overline{2w}\Lambda \tilde{c}\alpha\right) \bigg| \le \sum\limits_{\substack{w\bmod p^t\\ (w,p)=1}}\bigg| \sum\limits_{\substack{\alpha\in \mathbb{Z}\\ 
\alpha\in I_w}} 
e_{p^{t}}\left(\overline{2w}\Lambda \tilde{c}\alpha\right) \bigg|,
\end{split}
\end{equation*}
where  $I_w$ is an interval of length $K/(\tilde{c}p^t)$ depending on $w$.
By a standard bound for linear exponential sums (see \cite[(8.6)]{IwKo}, for example), the last double sum is bounded by 
\begin{equation*}
\sum\limits_{\substack{w\bmod p^t\\ (w,p)=1}}\Big| \sum\limits_{\substack{\alpha\in \mathbb{Z}\\ 
\alpha\in I_w}} 
e_{p^{t}}\left(\overline{2w}\Lambda \tilde{c}\alpha\right) \Big|\ll \sum\limits_{\substack{w\bmod p^t\\ (w,p)=1}} \left|\left| \frac{\overline{2w}\Lambda \tilde{c}}{p^t} \right|\right|^{-1}\\ = \sum\limits_{\substack{\tilde{w}\bmod p^t\\ (\tilde{w},p)=1}} \left|\left|\frac{\tilde{w}}{p^t}\right|\right|^{-1} \ll p^t\log p^t\ll p^{s/2}\log p^s, 
\end{equation*}
where for $x\in \mathbb{R}$, $||x||$ is the distance of $x$ to the nearest integer.
 
If $g>t$, then the outer sum over $w$ has $O(1)$ summands, and the inner sum has $O(1+K/(\tilde{c}p^g))$ summands. Hence, recalling $K\le p^s$ and $t=s/2$, the sum in question is trivially bounded by $O(p^{s/2})$ in this case. This completes the proof for the case when $s$ is even. If $s$ is odd, then we write $s=2t+1$ and take $w$ to be a solution of $w^2\equiv k\Lambda \bmod{p^{t+1}}$ with $w\equiv a\bmod{p}$. This can be lifted uniquely to a solution of $u^2\equiv k\Lambda \bmod{p^s}$ with $u\equiv a\bmod{p}$. Similar arguments as before give a bound of $O\left(p^{t+1}\log p^t\right)$ for the sum in question, from which the desired result follows. 
\end{proof}

We deduce the following from Proposition \ref{1stprop}.

\begin{corollary} \label{2stcor}
Let $p$ be a fixed odd prime, $\mu\in \{0,1\}$, $\Lambda\in \mathbb{Z}$ with $(\Lambda,p)=1$, $s\in \mathbb{N}$ with $s\ge 2$,  $b\in \mathbb{Z}$, $c\in \mathbb{N}$ and $0< K\le p^s$. Then
$$
\sum\limits_{\substack{0<k\le K\\ k\equiv b \bmod{c}\\ (k,p)=1}}\ \sum\limits_{\substack{u^2\equiv k\Lambda\bmod{p^s}}} \left(\frac{u}{p^s}\right)^{\mu}e_{p^s}(u) \ll_p p^{s/2}\log p^s.
$$ 
\end{corollary}

\begin{proof}
We have 
$$
\sum\limits_{\substack{0<k\le K\\ k\equiv b \bmod{c}\\ (k,p)=1}}\ \sum\limits_{\substack{u^2\equiv k\Lambda\bmod{p^s}}} \left(\frac{u}{p^{s}}\right)^{\mu}e_{p^s}(u)=
\sum\limits_{\substack{a\bmod{p}\\ (a,p)=1}} \left(\frac{a}{p^s}\right)^{\mu}\sum\limits_{\substack{0<k\le K\\ k\equiv b \bmod{c}\\ (k,p)=1}}\ \sum\limits_{\substack{u^2\equiv k\Lambda\bmod{p^s}\\ u\equiv a\bmod{p}}} e_{p^s}(u).
$$
Now the desired result follows by applying Proposition \ref{1stprop}.
\end{proof}

\subsubsection{Final evaluation}
Now we are ready to finalize our estimation of the error term $U$. We shall prove the following for the $k$-sums in \eqref{odderror} and \eqref{evenerror}. 

\begin{lemma}\label{1stlemma}
Let $p$ be a fixed odd prime, $\mu\in \{0,1\}$, $\Lambda\in \mathbb{Z}$ with $(\Lambda,p)=1$, $r\in \mathbb{N}$ with $0\le r\le m-2$ and $1\le K,N\le p^{m-r}$. Define $\tau_n(k)$ as in \eqref{taudef} and set $s:=m-r$. Then if $n\ge 6$, we have   
    \begin{equation*} 
\sum\limits_{\substack{0<k\leq K\\ (k,p)=1}}\tau_n(k) \sum\limits_{u^2\equiv k\Lambda\bmod{p^s}} \left(\frac{u}{p^s}\right)^{\mu}e_{p^s}(u) \ll \left(\left(\frac{p^s}{N}\right)^{n-2}p^{s/2}+K^{n/2}\left(\frac{p^s}{N}\right)^{(3-n)/2}+K\left(\frac{p^s}{N}\right)^{(n-1)/2}\right)p^{\varepsilon s}.
    \end{equation*}
\end{lemma}
\begin{proof}
Recalling \eqref{weight} and \eqref{taudef}, we have 
\begin{equation*} 
    \tau_n(k)=\sum_{\substack{{\bf l}\in\mathbb{Z}^n\\(l_1\cdots l_n,p)=1\\ \tilde{Q}({\bf l})=k}} w_P({\bf l}),
    \end{equation*}
where 
\begin{equation} \label{Pdef}
w_P({\bf x}):=\prod\limits_{i=1}^n\hat{\Phi}\left(\frac{x_i}{P}\right) \quad \mbox{with} \quad P=\frac{p^s}{N}.
\end{equation}
We note that the condition $N\le p^s$ in Lemma \ref{1stlemma} implies $P\ge 1$, which is a condition in Proposition \ref{modHeath}. 
Now using the said Proposition \ref{modHeath}, $\tau_n(k)$ satisfies the asymptotic formula
\begin{equation*} 
\tau_n(k)=\sigma_{\infty,P}(k)\sigma(k)P^{n-2}+O\left(\left(k^{n/2-1}P^{(3-n)/2}+P^{(n-1)/2}\right)(kP)^{\varepsilon}\right),
\end{equation*}
where 
$$
\sigma_{\infty,P}(k):=\lim\limits_{\varepsilon\rightarrow 0} \frac{1}{2\epsilon} \int\limits_{|F^{(0)}({\bf x})-k/P^2|< \epsilon} w_P({\bf x}){\rm d}{\bf x}
$$
and 
\begin{equation*}
    \sigma(k):=\sum_{q=1}^\infty a_q(k)
\end{equation*}
with 
$$
a_q(k):=\frac{1}{(pq)^n}\sum_{\substack{a\bmod q\\(a,q)=1}}\sum\limits_{\substack{{\bf x}\bmod pq\\ (x_1\cdots x_n,p)=1}}e_q\left(a(\tilde{Q}({\bf x})-k) \right).
$$

We observe that $a_q(k)$ just depends on the residue class of $k$ modulo $q$ and 
\begin{equation} \label{aqk}
\begin{split}
        |a_q(k)|= &\Bigg|\frac{1}{(pq)^n}\sum_{\substack{a\bmod q\\(a,q)=1}}e_q\left(-ak \right)\sum_{\substack{{\bf x}\bmod pq\\ (x_1\cdots x_n,p)=1}}e_q\left(a\tilde{Q}({\bf x})\right)\Bigg|\\
        \leq & \frac{1}{p^nq^{n-1}}\max_{\substack{a\bmod q\\(a,q)=1}}\left|\sum_{\substack{{\bf x}\bmod pq\\ (x_1\cdots x_n,p)=1}} e_q\left(a\tilde{Q}({\bf x}) \right)\right|\\
       \ll_{p,n} & q^{1-n/2},
\end{split}
\end{equation}
where the bound in the last line is obtained by factorizing the multiple exponential sum over ${\bf x}$ into a product of single variable quadratic exponential sums over $x_i$, detecting the coprimality condition $(x_i,p)=1$ using M\"obius inversion and using Proposition \ref{Gaussprop}(v) to bound the resulting complete quadratic Gauss sums.  
It follows that
\begin{equation} \label{ktrans}
\begin{split}
& \sum\limits_{\substack{0<k\leq K\\ (k,p)=1}}\tau_n(k) \sum\limits_{u^2\equiv k\Lambda\bmod{p^s}} \left(\frac{u}{p^s}\right)^{\mu}e_{p^s}(u)\\ =& \sum\limits_{\substack{0<k\leq K\\ (k,p)=1}} \sigma_{\infty,P}(k)\sigma(k)P^{n-2}\sum\limits_{u^2\equiv k\Lambda\bmod{p^s}} \left(\frac{u}{p^s}\right)^{\mu}e_{p^s}(u)+O\left(\left(K^{n/2}P^{(3-n)/2}+KP^{(n-1)/2}\right)(KP)^{\varepsilon}\right)\\
= & P^{n-2}\sum_{q=1}^\infty \sum_{b\bmod q}a_q(b)\sum_{\substack{0<k\leq K\\ k\equiv b\bmod q\\ (k,p)=1}} \sigma_{\infty,P}(k)\sum\limits_{u^2\equiv k\Lambda\bmod{p^s}} \left(\frac{u}{p^s}\right)^{\mu}e_{p^s}(u)+O\left(\left(K^{n/2}P^{(3-n)/2}+KP^{(n-1)/2}\right)(KP)^{\varepsilon}\right)
\end{split}
\end{equation}
if $n\ge 7$. (In this case, the sum over $q$ converges absolutely.) Removing the factor $\sigma_{\infty,P}(k)$ using partial summation and Lemma \ref{Heathrem}, we deduce from Corollary \ref{2stcor} that
\begin{equation} \label{kest}
\sum_{\substack{0<k\leq K\\ k\equiv b\bmod q\\ (k,p)=1}} \sigma_{\infty,P}(k)\sum\limits_{u^2\equiv k\Lambda\bmod{p^s}} \left(\frac{u}{p^s}\right)^{\mu}e_{p^s}(u)
\ll p^{s/2}\left(Kp^s\right)^{\varepsilon}.
\end{equation}
Combining \eqref{aqk}, \eqref{ktrans} and \eqref{kest}, we obtain
\begin{equation} \label{n7}
\begin{split}
& \sum\limits_{\substack{0<k\leq K\\ (k,p)=1}}\tau_n(k) \sum\limits_{u^2\equiv k\Lambda\bmod{p^s}} \left(\frac{u}{p^s}\right)^{\mu}e_{p^s}(u)\\ 
\ll & \left(P^{n-2}\sum\limits_{q=1}^{\infty} q^{2-n/2} p^{s/2}+K^{n/2}P^{(3-n)/2}+KP^{(n-1)/2}\right)\left(KPp^s\right)^{\varepsilon}\\ \ll & \left(P^{n-2}p^{s/2}+K^{n/2}P^{(3-n)/2}+KP^{(n-1)/2}\right)\left(KPp^s\right)^{\varepsilon},
\end{split}
\end{equation}
provided that $n\ge 7$. This yields the desired result in this case on recalling the definition on $P$ in \eqref{Pdef} and the condition $1\le K\le p^s$.

If $n=6$, then using \eqref{aqk}, we approximate $\sigma(k)$ by a finite sum in the form 
\begin{equation}
    \sigma(k)=\sum_{q\leq L}a_q(k)+O(L^{-1}),
\end{equation}
where $L\ge 1$ is a free parameter which we fix later. In this case, along similar lines as above, we get the bound
\begin{equation} \label{similar}
\begin{split}
& \sum\limits_{\substack{0<k\leq K\\ (k,p)=1}}\tau_n(k) \sum\limits_{u^2\equiv k\Lambda\bmod{p^s}} \left(\frac{u}{p^s}\right)^{\mu}e_{p^s}(u)\\ 
\ll & \left(P^{4}\sum\limits_{q\le L} q^{-1} p^{s/2}+K^3P^{-3/2}+KP^{5/2}\right)\left(KPp^s\right)^{\varepsilon}+L^{-1}\left(K^3P^{-3/2}+KP^4\right)(KP)^{\varepsilon}\\ \ll & \left(P^4p^{s/2}+K^3P^{-3/2}+KP^{5/2}+L^{-1}K^3P^{-3/2}+L^{-1}KP^4\right)\left(KLPp^s\right)^{\varepsilon},
\end{split}
\end{equation}
where we use the bound 
\begin{equation*}
\begin{split}
\tau_n(k)\ll & \left|\sigma_{\infty,P}(k)\sigma(k)P^{n-2}\right|+\left(k^{n/2-1}P^{(3-n)/2}+P^{(n-1)/2}\right)(kP)^{\varepsilon}\\ 
\ll & \left(k^{n/2-1}P^{(3-n)/2}+P^{n-2}\right)(kP)^{\varepsilon}\\ = &\left(k^2P^{-3/2}+P^4\right)(kP)^{\varepsilon}
\end{split}
\end{equation*}
to obtain the last term containing $L^{-1}$ in the second line of \eqref{similar}. 
Taking $L:=P^{3/2}$ gives the same bound as in \eqref{n7} in the case $n=6$, which completes the proof.
\end{proof}

In the following, we complete our estimation of the error term $U$. We can apply Lemma \ref{1stlemma} to estimate the inner-most double sums over $k$ and $u$ in \eqref{odderror} and \eqref{evenerror} if $1\le N,K\le p^{m-r}$. Recalling \eqref{Kdef} and \eqref{Mrdef}, this is the case if 
$$
Cp^{(m-r+\varepsilon m)/2}\le N\le  p^{m-r},  
$$
for a suitable constant $C>0$ depending on the coefficients of our quadratic form $Q$.
The above is equivalent to 
$$
C^2\frac{p^{(1+\varepsilon)m}}{N^2}\le p^r\le \frac{p^m}{N}.
$$
If $N\ge Cp^{(1+\varepsilon)m/2}$, then the lower bound above can be omitted. Indeed, this is ensured by the condition $N\ge q^{1/2+\varepsilon}=p^{(1/2+\varepsilon)m}$ in Theorem \ref{thm1}, provided $m$ is large enough. Now using Lemma \ref{1stlemma}, \eqref{Kdef} and \eqref{Mrdef}, the contributions $U^+$ of all $r$'s satisfying $p^r\le p^m/N$ to the right-hand sides of \eqref{odderror} and \eqref{evenerror} are bounded by  
\begin{equation*} 
\begin{split}
    U^+ \ll & \frac{N^n}{p^{m(n+1)/2}}\sum_{\substack{r\ge 0\\ p^r\le p^m/N}} p^{(n-1)r/2} \cdot \left(\left(\frac{p^{m-r}}{N}\right)^{n-2}p^{(m-r)/2}+\left(\frac{p^{m-r}}{N}\right)^{(n+3)/2}\right)p^{\varepsilon m}\\
\ll & \frac{N^n}{p^{m(n+1)/2}} \cdot \left(\left(\frac{p^{m}}{N}\right)^{n-2}p^{{m}/2}+\left(\frac{p^{m}}{N}\right)^{(n+3)/2}\right)p^{\varepsilon m}\\
= & \left(N^2p^{m(n/2-2)}+N^{(n-3)/2}p^m\right)p^{\varepsilon m}.
\end{split}
\end{equation*}

If $p^m/N<p^r\le p^{(1+\varepsilon)m}/N$, then in view of \eqref{weight} and \eqref{taudef}, we have $\tau_n(k)=O(1)$ if $|k|\ll p^{\varepsilon m}$, and $\tau_n(k)$ is negligible if $|k|\gg p^{\varepsilon m}$. Here we use the rapid decay of $\hat{\Phi}$.
Moreover, if $N\ge p^{(1/2+\varepsilon)m}$, then $p^r\le p^{(1+\varepsilon)m}/N$ implies $r\le m/2$. Hence, the contributions $U^{-}$ of all $r$'s satisfying $p^m/N<p^r\le p^{(1+\varepsilon)m}/N$ to the right-hand sides of \eqref{odderror} and \eqref{evenerror} are bounded by 
\begin{equation*} \label{U-}
U^-\ll \frac{N^n}{p^{m(n+1)/2}} \cdot p^{(n-1)m/4}\cdot p^{\varepsilon m} =\frac{N^n}{p^{m(n+3)/4}}\cdot p^{\varepsilon m}.
\end{equation*}

Finally, the contribution of $r$'s such that $p^r> p^{(1+\varepsilon)m}/N$ is $O(1)$ since $\tau_n(k)$ is negligible for all $k$ in this case. Hence,
$$
U=U^++U^-+O(1)\ll \left(N^2p^{m(n/2-2)}+N^{(n-3)/2}p^m+N^np^{-m(n+3)/4}\right)p^{\varepsilon m}.
$$

Lastly, we compare the orders of magnitude of the error term $U$ and the main term $T_0$,  as given in in \eqref{maintermev}. We have $U\ll p^{-\varepsilon m}T_0$ if 
$$
N^2p^{m(n/2-2)}\ll p^{-2\varepsilon m}N^np^{-m}, \quad N^{(n-3)/2}p^m\ll p^{-2\varepsilon m}N^np^{-m}, \quad N^np^{-m(n+3)/4}\ll p^{-2\varepsilon m}N^np^{-m},
$$
which is the case if $N\ge p^{(1/2+\varepsilon)m}$ and $n\ge 5$.
This completes the proof of Theorem \ref{thm1}. 

\section{Proof of Theorem \ref{thm1'} (homogeneous congruences)}
Using \cite[Remark 1, equation (5.1), equation following (5.1), equations following (5.8)]{BaBaHa}, we have 
\begin{equation} \label{secondasymp}
\sum\limits_{\substack{(x_1,...,x_{n})\in \mathbb{Z}^{n}\\ (x_1,...,x_n)\not\equiv (0,...,0)\bmod p\\ Q(x_1,...,x_n)\equiv 0 \bmod{q}\\
}}
\prod_{i=1}^{n}\Phi\left(\frac{x_i}{N}\right)=A_p(Q)\cdot
\hat{\Phi}(0)^{n}\cdot \frac{N^{n}}{q}+U,
\end{equation}
where the error term $U$ is bounded by
\begin{equation} \label{seconderror}
 U\ll\frac{N^{n}}{p^{m(n+1)}}\sum_{r=0}^{m-2}p^{(m+r)n/2+(m-r)}\sum\limits_{\substack{{\bf l}\in \mathbb{Z}^{n}\\|l_1|,...,|l_{n}|\le L_r\\\tilde{Q}({\bf l})\equiv 0 \bmod p^{m-r-2}}} 1+O(1).
\end{equation}
Here,
\begin{equation} \label{Lrdef}
L_r:=\frac{p^{m-r+\varepsilon m}}{N}
\end{equation}
and 
$$
\tilde{Q}({\bf x})=\overline{\lambda_1}x_1^2+\cdots +\overline{\lambda_n}x_n^2, 
$$
where $\lambda_i\overline{\lambda_i}\equiv 1 \bmod{p^m}$. We prove the following lemma, from which a satisfactory estimate for $U$ can be deduced easily.

\begin{lemma} \label{Sbound}
Let $\varepsilon>0$, $p$ be an odd prime, $s\in \mathbb{N}$, $c:=p^s$, $b\in\mathbb{Z}$, $\alpha_1,...,\alpha_4\in \mathbb{Z}$ with $(\alpha_1\alpha_2\alpha_3\alpha_4,p)=1$ and $M\ge 1$ with $8M^2<c$. Then 
$$
\sum\limits_{\substack{(l_1,l_2,l_3,l_4)\in \mathbb{Z}^4\\ |l_1|,|l_2|,|l_3|,|l_4|\le M\\ \alpha_1l_1^2+\alpha_2l_2^2+\alpha_3l_3^2+\alpha_4l_4^2\equiv b\bmod{c}}}1 \ll M^{2+\varepsilon}.
$$   
\end{lemma}

\begin{proof}
We write 
\begin{equation*}
\begin{split}
\sum\limits_{\substack{(l_1,l_2,l_3,l_4)\in \mathbb{Z}^4\\ |l_1|,|l_2|,|l_3|,|l_4|\le M\\ \alpha_1l_1^2+\alpha_2l_2^2+\alpha_3l_3^2+\alpha_4l_4^2\equiv b\bmod{c}}}1
 = \sum\limits_{a\bmod{c}}
\Bigg(\sum\limits_{\substack{(l_1,l_2)\in \mathbb{Z}^2\\ |l_1|,|l_2|\le M\\ \alpha_1 l_1^2+\alpha_2 l_2^2\equiv a \bmod{c}}} 1\Bigg)\Bigg(\sum\limits_{\substack{(l_3,l_4)\in \mathbb{Z}^2\\ |l_3|,|l_4|\le M\\ -\alpha_3l_3^2-\alpha_4 l_4^2+b\equiv a \bmod{c}}} 1\Bigg).
\end{split}
\end{equation*}
Using the Cauchy-Schwarz inequality, the right-hand side is bounded by
$$
\ll (S_1S_2)^{1/2}, \mbox{ where } S_1:=\sum\limits_{a\bmod c} \Bigg(\sum\limits_{\substack{(l_1,l_2)\in \mathbb{Z}^2\\ |l_1|,|l_2|\le M\\ \alpha_1 l_1^2+\alpha_2 l_2^2\equiv a \bmod{c}}} 1\Bigg)^2 \mbox{ and } S_2:=\sum\limits_{a\bmod c}\Bigg(\sum\limits_{\substack{(l_3,l_4)\in \mathbb{Z}^2\\ |l_3|,|l_4|\le M\\ -\alpha_3l_3^2-\alpha_4 l_4^2+b\equiv a \bmod{c}}} 1\Bigg)^2.
$$

Clearly,
$$
S_1=\sum\limits_{\substack{(l_1,l_2,m_1,m_2)\in \mathbb{Z}^4\\ |l_1|,|l_2|,|m_1|,|m_2|\le M\\ \alpha_1l_1^2+\alpha_2l_2^2\equiv \alpha_1m_1^2+\alpha_2m_2^2\bmod{c}}} 1.  
$$
We divide this into
$$
S_1=D+E,
$$
where 
$$
D:=\sum\limits_{\substack{(l_1,l_2,m_1,m_2)\in \mathbb{Z}^4\\ |l_1|=|m_1|\le M\\ |l_2|=|m_2|\le M}} 1
$$
and 
$$
E:=\sum\limits_{\substack{(l_1,l_2,m_1,m_2)\in \mathbb{Z}^4\\ |l_1|,|l_2|,|m_1|,|m_2|\le M\\ |l_1|\not=|m_1| \ \text{or}\ |l_2|\not=|m_2|\\ 
\alpha_1l_1^2+\alpha_2l_2^2\equiv \alpha_1m_1^2+\alpha_2m_2^2\bmod{c}}} 1.
$$
Obviously, 
$$
D\ll M^2.
$$
Further, we write $E$ as 
$$
E = \sum\limits_{\substack{(l_1,l_2,m_1,m_2)\in \mathbb{Z}^4\\ |l_1|,|l_2|,|m_1|,|m_2|\le M\\ \alpha_1(l_1-m_1)(l_1+m_1)\equiv \alpha_2(m_2-l_2)(m_2+l_2) \bmod{c}\\ (l_1-m_1)(l_1+m_1)\not=0 \ \text{or}\  (m_2-l_2)(m_2+l_2)\not=0}} 1.
$$
We observe that if one of the numbers 
$$
A_1=(l_1-m_1)(l_1+m_1) \quad \mbox{and} \quad A_2=(m_2-l_2)(m_2+l_2)
$$
in the summation condition above equals 0, then the other one equals 0 as well. To see this, recall $c=p^s$ and note that if $\alpha_i A_i\equiv 0 \bmod{p^s}$, then $p^u|(l_i-m_i)$ and $p^v|(l_i+m_i)$ with $u+v\ge s$, which is not possible if $8M^2< c$ unless $l_i-m_i=0$ or $l_i+m_i=0$. 
It follows that
\begin{equation*}
\begin{split}
E\ll & \sum\limits_{\substack{(l_1,l_2,m_1,m_2)\in \mathbb{Z}^4\\ |l_1|,|l_2|,|m_1|,|m_2|\le M\\ \alpha_1(l_1-m_1)(l_1+m_1)\equiv \alpha_2(m_2-l_2)(m_2+l_2) \bmod{c}\\ (l_1-m_1)(l_1+m_1)\not=0 \ \text{and}\ (m_2-l_2)(m_2+l_2)\not=0}} 1\\ 
 \ll & \sum\limits_{\substack{(A_1,A_2)\in \mathbb{Z}^2\\ 0<|A_1|,|A_2|\le 4M^2\\ \alpha_1A_1\equiv \alpha_2A_2\bmod{c}}} d(|A_1|)d(|A_2|),
\end{split}
\end{equation*}
where $d(k)$ denotes the number of divisors of $k\in \mathbb{N}$. Taking the condition $(\alpha_1\alpha_2,c)=1$ into account, we observe that if $8M^2<c$, then any given $A_2$ in the summation condition above fixes $A_1$, if it exists at all.  Hence, using the well-known divisor bound $d(k)\ll k^{\varepsilon}$, we trivially get
\begin{equation*}
E\ll M^{2+\varepsilon}.
\end{equation*}

Collecting everything above, we obtain the bound
$$
S_1\ll M^{2+\varepsilon}.
$$
The same arguments as above apply to $S_2$, getting
$$
S_2\ll M^{2+\varepsilon},
$$
and hence the desired result follows. 
\end{proof}

Now if $8L_r^2\le p^{m-r-2}$, then we may apply Lemma \ref{Sbound} above to bound the inner-most sum over ${\bf l}$ in \eqref{seconderror} by 
\begin{equation} \label{ulti}
\begin{split}
\sum\limits_{\substack{{\bf l}\in \mathbb{Z}^{n}\\|l_1|,...,|l_{n}|\le L_r\\\tilde{Q}({\bf l})\equiv 0 \bmod p^{m-r-2}}} 1 = & \sum\limits_{\substack{(l_5,...,l_n)\in \mathbb{Z}^{n-4}\\
|l_5|,...,|l_{n}|\le L_r}}\sum\limits_{\substack{(l_1,...,l_4)\in \mathbb{Z}^{4}\\
|l_1|,...,|l_{4}|\le L_r\\ \overline{\lambda_1}l_1^2+\cdots +\overline{\lambda_4}l_4^2\equiv -(\overline{\lambda_5}l_5^2+\cdots +\overline{\lambda_n}l_n^2) \bmod{p^{m-r-2}}}} 1\\
\ll & \sum\limits_{\substack{(l_5,...,l_n)\in \mathbb{Z}^{n-4}\\
|l_5|,...,|l_{n}|\le L_r}} L_r^{2+\varepsilon} \\
\ll & L_r^{n-2+\varepsilon},
\end{split}
\end{equation}
where it is understood that the outer sum over $(l_5,...,l_n)$ is empty and $\overline{\lambda_5}l_5^2+\cdots +\overline{\lambda_n}l_n^2=0$ if $n=4$. It is easy to check that indeed 
$8L_r^2\le p^{m-r-2}$ if 
\begin{equation} \label{stronger}
N\ge p^{(1/2+2\varepsilon)m}=q^{1/2+2\varepsilon}
\end{equation}
and $m$ is large enough: We have 
$$
8L_r^2\le p^{m-r-2}\Longleftrightarrow 8\cdot \left(\frac{p^{m-r+\varepsilon m}}{N}\right)^2\le p^{m-r-2} \Longleftrightarrow N\ge 8^{1/2}p^{m/2-r/2+\varepsilon m+1}
$$
and 
$$
8^{1/2}p^{m/2-r/2+\varepsilon m+1}\le 8^{1/2}p^{m/2+\varepsilon m+1}\le p^{(1/2+2\varepsilon)m}
$$
for large enough $m$. 

Combining \eqref{seconderror}, \eqref{Lrdef} and \eqref{ulti}, we obtain
$$
 U\ll p^{\varepsilon m}\cdot \frac{N^{n}}{p^{m(n+1)}}\sum_{r=0}^{m-2}p^{(m+r)n/2+(m-r)} \cdot \left(\frac{p^{m-r}}{N}\right)^{n-2} \ll p^{\varepsilon m}\cdot \frac{N^{n}}{p^{m(n+1)}}\cdot  p^{mn/2+m} \cdot \left(\frac{p^{m}}{N}\right)^{n-2}\ll N^2p^{(n/2-2+\varepsilon)m}
$$
if $n\ge 4$. 

It remains to compare this bound with the size of the main term. We have
$$
N^2p^{(n/2-2+\varepsilon)m}\ll p^{-\varepsilon m} T_0
$$
if 
$$
N^2p^{(n/2-2+\varepsilon)m}\le p^{-\varepsilon m} \cdot \frac{N^n}{p^m},
$$
which is the case if $n\ge 4$ and $N\ge p^{(1/2+\varepsilon)m}=q^{1/2+\varepsilon}$. In \eqref{stronger}, we imposed the stronger condition $N\ge q^{1/2+2\varepsilon}$, but we are free to redefine $2\varepsilon$ as $\varepsilon$. This establishes the result of Theorem \ref{thm1'}.

\end{document}